\documentclass[a4paper,11pt]{article}
\usepackage[utf8]{inputenc}
\usepackage{amsfonts, amsmath, amssymb, amsthm, mathrsfs, mathtools}
\usepackage{geometry}
\usepackage{xcolor,soul}
\usepackage{float}
\usepackage{varwidth}
\usepackage[shortlabels]{enumitem}
\usepackage{framed}
\usepackage[ruled,vlined,algo2e]{algorithm2e}
\usepackage{algorithm, algorithmic}
\usepackage{tcolorbox}
\usepackage{graphicx, subcaption, mwe}
\usepackage[justification=centering, singlelinecheck=off]{caption}
\usepackage{hyperref}
\usepackage{todonotes}
\usepackage{tabularx, booktabs, makecell, caption, multirow, multicol}
\usepackage{authblk}

\binoppenalty=\maxdimen 
\relpenalty=\maxdimen 
 
\hypersetup{
        colorlinks=true,   
        linkcolor=blue,    
        citecolor=blue,    
        urlcolor=blue
}

\theoremstyle{definition}
\newtheorem{theorem}{Theorem}[section]
\newtheorem{lemma}[theorem]{Lemma}
\newtheorem{proposition}[theorem]{Proposition}
\newtheorem{corollary}[theorem]{Corollary}
\theoremstyle{definition}

\newtheorem{assumption}[theorem]{Assumption}
\theoremstyle{remark}
\newtheorem{remark}[theorem]{Remark}
\newtheorem{example}[theorem]{Example}

\DeclareMathOperator{\prox}{\mathbf{prox}}
\DeclareMathOperator{\Hilbert}{\mathcal{H}}
\DeclareMathOperator{\relint}{ri}

\DeclareMathOperator{\dist}{dist}
\DeclareMathOperator{\dom}{dom}
\DeclareMathOperator{\range}{range}
\DeclareMathOperator{\conv}{conv}
\DeclareMathOperator*{\argmax}{argmax}
\DeclareMathOperator*{\argmin}{argmin}

\makeatletter
\newcommand\Prob[3]{%
    \textbf{w.p.} #1 \textbf{do}%
    \algocf@group{#2}%
    \textbf{Otherwise}%
    \algocf@group{#3}%
}
\makeatother

\title{Identification of Active Subfunctions in Finite-Max Minimisation via a Smooth Reformulation}
\author[1]{Charl J. Ras}
\author[1]{Matthew K. Tam}
\author[1]{Daniel J. Uteda}
\affil[1]{School of Mathematics \& Statistics, The University of Melbourne, Parkville VIC 3010, Australia.}
\date{\today}

\begin{document}
\maketitle

\begin{abstract}
In this work, we consider a nonsmooth minimisation problem in which the objective function can be represented as the maximum of finitely many smooth ``subfunctions''. First, we study a smooth min-max reformulation of the problem. Due to this smoothness, the model provides enhanced capability of exploiting the structure of the problem, when compared to methods that attempt to tackle the nonsmooth problem directly. Then, we present several approaches to identify the set of active subfunctions at a minimiser, all within finitely many iterations of a first order method for solving the smooth model. As is well known, the problem can be equivalently rewritten in terms of these subfunctions, but a key challenge is to identify this set \textit{a priori}. Such an identification is clearly beneficial in an algorithmic sense, since one can apply this knowledge to create an equivalent problem with lower complexity, thus facilitating generally faster convergence. Finally, numerical results comparing the accuracy of each of these approaches are presented, along with the effect they have on reducing the complexity of the original problem.

\end{abstract}

\providecommand{\keywords}[1]
{
  \small	
  \noindent \textbf{Keywords.} #1
}

\sethlcolor{cyan}
\section{Introduction}
Let $\Hilbert$ be a real and finite dimensional Hilbert space. We consider minimisation problems of the form
\begin{equation}\label{eq:minmax}
    \min_{x\in\Hilbert} f(x),\text{~where~} f(x) := \max_{i\in I}\{f_i(x)\}
\end{equation}
Here $I := \{1,\dots,N\}$ denotes a finite index set, and the subfunctions $f_i\colon\Hilbert\to\mathbb{R}$ are assumed convex and differentiable with locally Lipschitz gradients. Problems of the form~\eqref{eq:minmax} have various applications, including in neural networks~\cite{deodhare1996synthesis, manevitz1996approx, ZAMIR2015947}, signal processing~\cite{liu2011maxmin, yamamoto2003optimal, nagahara2011MinMax}, optimal recovery~\cite{binev2017data, cohen2020optimal}, and facility location~\cite{elzinga1976minimax, elzinga1972geometrical, elzinga1972minimum, love1973multi}. In this work, we are focused on first order methods which can exploit the underlying structure of~\eqref{eq:minmax}, that is, via direct computation of $\nabla f_i$ for certain subfunctions $f_i$.

\paragraph{Approaches to solving \eqref{eq:minmax}.} A common first-order method for solving~\eqref{eq:minmax} is the \textit{subgradient method}, a generalisation of gradient descent to (potentially) nonsmooth problems.  Given an initial point $x^0\in\Hilbert$ and a step-size sequence $(\lambda_k)\subset\mathbb{R}_{++}$, this method takes the form
 \begin{equation}\label{eq:subgrad}
 x^{k+1} = x^k-\lambda_kg^k,  \quad\forall k\in\mathbb{N}.
 \end{equation}
where $g^k$ denotes any \emph{subgradient} of $f$ at $x^k$. In the setting of \eqref{eq:minmax}, the \emph{subdifferential} (\emph{i.e.,}~the set of subgradients) at a point is given by the convex hull of the gradients of the active subfunctions at said point, and hence \eqref{eq:subgrad} can be implemented with relative ease. Despite this, the subgradient method suffers from a number of unfavourable properties which can impact its numerical performance: the step-size sequence is typically required to decay to zero leading to slow convergence \cite[Chapter 2]{Shor1985}, and the chosen subgradients need not be descent directions (see \cite[Chapter 7]{ruszczynski2006nonlinear} or \cite[Chapter 2]{Shor1985}).

Another first-order method for solving \eqref{eq:minmax} involves applying the \textit{Primal-Dual Hybrid Gradient (PDHG)} algorithm~\cite{chambolle2016introduction, condat2023randprox} to the epigraphical formulation of~\eqref{eq:minmax}, as was done in \cite{cohen2020optimal}. In this approach, the nonsmooth objective function is converted to a system of $N$ smooth constraints through the addition of an auxiliary variable. These constraints manifest themselves in the PDHG iteration through their nearest point projectors. However, since the projection onto epigraphical constraints do not typically admit a closed form, in general, computing the projections requires solving an additional smooth subproblem.

Other frameworks for solving \eqref{eq:minmax} include those based on smoothing~\cite{Polak2003, polak2003algorithms, xingsi1992entropy, xu2001smoothing}, cutting planes~\cite{gaudioso2006incremental}, and $\mathcal{VU}$-decompositions \cite{mifflin2005VU, hare2020derivative}. However, the implementations of many of these methods are not straightforward, and typically involve solving one or more potentially difficult subproblems at each iteration.

\paragraph{Identification of active subfunctions.} 

One factor influencing the complexity of \eqref{eq:minmax} is the size $N$ of the index set $I$. Indeed, the objective is smooth when $N=1$, and any proper, convex, and lower semicontinuous (lsc) objective can be represented as the maximum of arbitrarily many subfunctions.
To solve~\eqref{eq:minmax}, it might therefore be advantageous to replace the index set $I$ with a subset thereof, to form a reduced problem with the same minimisers as in~\eqref{eq:minmax}, and solve this (potentially simpler) reduced problem instead. Within this paradigm, a natural choice for this subset would be the subset of those subfunctions which are active at a solution of \eqref{eq:minmax}, however this is usually not known \emph{a priori}. 

Identification of active subfunctions is closely related to \emph{identification of active constraints}. In the context of constrained minimisation, the latter refers to determining which inequality constraints are satisfied with equality at a solution. When the objective function is continuously differentiable and the constraints are linear,  Burke and Mor{\'e} \cite{burke1988identification} showed that the active constraints can be identified using the sequences generated by certain iterative algorithms. More precisely, if a sequence (\emph{e.g.,} generated by projected gradient descent) converges to a stationary point of the problem, it was shown that the sequence and the stationary point are eventually contained within the relative interior of the same face of the polyhedron defined by the linear constraints. This result was later extended to convex non-linear constraints by Wright~\cite{wright1993identifiable}.

Using the notion of \emph{partially smooth functions}, Hare and Lewis~\cite{hare2004identifying} established a general result concerning identification of function activity, relative to certain manifolds, by a sequence of subgradients which converges to zero. Although this framework is able to model finite-max functions~\cite[Example~2.5]{hare2004identifying}, the conditions required to apply the identification result \cite[Theorem 5.3]{hare2004identifying} to the minimisation problem \eqref{eq:minmax} are restrictive. In particular, partial smoothness of $f$ at a solution $x^*$ of \eqref{eq:minmax} requires linear independence of $S:=\{\nabla f_i(x^*)\colon f(x^*)=f_i(x)\}$ which conflicts with the first order optimally condition of zero being contained in the convex hull of $S$. A further difficulty in applying this type of result is that, for first order methods, it is generally not possible to guarantee that subgradients generated by the method converge to zero.

A different approach to constraint identification was studied in \cite{facchinei1998accurate, oberlin2006active}. Rather than considering the original minimisation problem directly, constraint identification was examined in the context of the saddle point problem associated with the Lagrangian function. Within this framework, the identification of active constraints is made possible through the general notion known as \textit{identification functions}. In contrast to \cite{burke1988identification, wright1993identifiable}, this approach is more flexible in terms of the method applied to generate sequences, and doesn't require additional assumptions such as nondegeneracy.

\paragraph{Our Contributions.} 
In this work, we consider an approach to solving~\eqref{eq:minmax} that can exploit the differentiability of the subfunctions in~\eqref{eq:minmax}, and we discuss the finite time identification results which can be derived within this model. More precisely, we consider an approach based on reformulating the nonsmooth minimisation problem~\eqref{eq:minmax} as the smooth saddle point problem given by
\begin{equation}\label{eq:saddle}
\min_{x\in\Hilbert}\max_{y\in\Delta^N}\sum_{i\in I} y_i f_i(x)~=: \phi(x,y),
\end{equation}
where $\Delta^N=\left\{y\in\mathbb{R}^N_+\colon\sum_{i=1}^N y_i=1\right\}$ denotes the probability simplex. This formulation has two desirable properties: (i) its objective function is differentiable whenever the subfunctions in \eqref{eq:minmax} are differentiable, and (ii) its constraint set $\Delta^N$ is instance independent, and the projection is readily computable in $\mathcal{O}(N\log N)$ time \cite{condat2016fast}. In contrast, the complexity of the constraint set in the epigraphicial formulation depends on the problem instance. 

We then investigate several approaches to identify the set of active subfunctions at a minimiser. Part of these results are obtained by directly studying the problems~\eqref{eq:minmax} and~\eqref{eq:saddle}, as opposed to the aforementioned existing works which study more general nonsmooth functions (\emph{eg}, \cite{hare2004identifying, hare2007identifying}), and those that study constraint identification for a general constrained minimisation problem (\emph{eg}, \cite{burke1988identification, facchinei1998accurate, oberlin2006active, wright1993identifiable}). The remainder are based on the original work of \cite{facchinei1998accurate} for active constraint identification.

The rest of this paper is organised as follows. In Section~\ref{sec:prelim}, we recall various definitions and discuss KKT theory. In Section~\ref{sec:saddle}, we discuss the smooth model~\eqref{eq:saddle}, including a proof of equivalence in Theorem~\ref{thm:reform}. Section~\ref{sec:identification} contains results on finite time support identification, the main results being Theorem~\ref{thm:subgradient} and Corollary~\ref{cor:eps-subdiff}.  Finally, in Section~\ref{sec:exp}, we compare the accuracy of the results from Section~\ref{sec:identification}, and demonstrate how they can be applied to improve computational performance.

\section{Preliminaries}\label{sec:prelim}

\paragraph{Notation \& Definitions.} Throughout, $\Hilbert$ denotes a real finite dimensional Hilbert space equipped with inner product $\langle\cdot,\cdot\rangle$ and induced norm $\|\cdot\|$. The \textit{subdifferential} of a convex function $f\colon\Hilbert\to(-\infty,+\infty]$ for $x_0\in\dom f:= \{x\in\Hilbert\colon f(x)<+\infty\}$ is defined as
$$\partial f(x_0) = \{\xi\in\Hilbert\colon\langle f(x) - f(x_0) - \langle\xi,x-x_0\rangle\geq 0~\forall x\in\Hilbert\},$$
and $\partial f(x_0) = \emptyset$ for $x_0\notin\dom f$. Its elements are called \textit{subgradients}. In particular, the following properties of $\partial f$ will be particularly useful in this paper.
\begin{lemma}\label{lem:subdiff}
    \begin{enumerate}[(a)]
        \item\label{item:opt} \cite[Theorem 16.3]{Bauschke2017} For $f\colon\Hilbert\to(-\infty,+\infty]$ proper, $x^*\in\argmin f$ if and only if $0\in\partial f(x^*)$.
        \item\label{item:max-rule} \cite[Theorem 18.5]{Bauschke2017} If $f=\max\{f_1,\dots,f_N\}$ as defined in~\eqref{eq:minmax}, then for any $x_0\in\Hilbert$,
        $$\partial f(x_0) = \conv\{\nabla f_i(x_0)~|~i\in I(x_0)\},$$
        where $I(x_0) := \{i\in I\colon f_i(x_0) = f(x_0)\}$ is the set of active subfunctions.
    \end{enumerate}
\end{lemma}

 Similarly to $\partial f$, the $\varepsilon$-\textit{subdifferential} for $x_0\in\dom f$ and given $\varepsilon>0$ is defined as 
$$\partial_\varepsilon f(x_0) = \{\xi\in\Hilbert\colon\langle f(x) - f(x_0) - \langle\xi,x-x_0\rangle+\varepsilon\geq 0~\forall x\in\Hilbert\},$$
and $\partial_{\varepsilon} f(x_0) = \emptyset$ for $x_0\notin\dom f$. Its elements are called $\varepsilon$-\textit{subgradients}. In regards to the $\varepsilon$-subdifferential, we will make use of the following result.
\begin{theorem}[Br{\o}ndsted--Rockafellar Theorem {\cite[Theorem 4.3.2]{borwein2010convex}}]\label{thm:BR}
    Let $f\colon\Hilbert\to(-\infty,+\infty]$ be a proper, convex, and lsc function. For any $\varepsilon,\delta\in\mathbb{R}_+, x\in\dom f$, and $\xi\in\partial_\varepsilon f(x)$, there exists $u\in\dom f$ and $\nu\in\partial f(u)$ such that
    $$\|u-x\|\leq \frac{\varepsilon}{\delta},\text{~and~}\|\nu-\xi\|\leq \delta.$$
\end{theorem}

The \textit{indicator function} $\iota_C\colon\Hilbert\to(-\infty,+\infty]$ of a set $C\subseteq\Hilbert$ takes the value $0$ for $x\in C$ and $+\infty$ otherwise. The \textit{normal cone} of a set $C$ at a point $x_0\in C$ is the subdifferential of its indicator function and is denoted $N_C(x_0) = \{\xi\in\Hilbert\colon\langle\xi,x-x_0\rangle\geq 0~\forall x\in C\}$. We will make use of the point-to-set distance $d(x,C):=\inf_{u\in C}\|x-u\|$, and $\relint C$ for the \textit{relative interior} of the set $C\subseteq\Hilbert$.

We denote $e=(1,\dots,1)^\top\in\mathbb{R}^n$, where the dimension will be clear from context. Then $\Delta^n=\left\{y\in\mathbb{R}^n_+\colon \langle e,y\rangle = 1\right\}$ is called the \textit{unit/probability simplex}, and the convex hull of a set of vectors $\{v_1,\dots,v_n\}$ is defined as 
$$\conv\{v_1,\dots,v_n\} := \left\{\sum_{i=1}^n\alpha_i v_i~|~ \alpha\in\Delta^n\right\}.$$
In particular, we will apply following results.

\begin{lemma}[\protect{\cite[Exercise 3.1]{Broendsted1983}}]\label{lem:rintCH}
For $K$ = $\conv\{v_1, \dots, v_n\}$, we have $x\in\relint K$ if and only if there exists $\alpha\in\Delta^n\cap\mathbb{R}_{++}$ such that $x = \sum_{i=1}^n \alpha_i v_i$.
\end{lemma}
 
\begin{lemma}\label{lem:cone} Let $y\in\Delta^n$. Then
\begin{enumerate}[(\alph*)]
    \item\label{item:normal} 
    $N_{\Delta^n}(y) = \{ \tau e - w~|~\tau\in\mathbb{R}, w\in\mathbb{R}^n_+\text{~s.t.~} w_i y_i = 0~\forall i\}.$
    \item\label{item:rintnormal}
    $\relint N_{\Delta^n}(y) = \{ \tau e - w~|~\tau\in\mathbb{R}, w\in\mathbb{R}^n_{+}\text{~s.t.~} w_i = 0 \text{~if and only if~} y_i > 0~\forall i\}.$
\end{enumerate}

\end{lemma}
\begin{proof}
Since $\Delta^n$ is polyhedral, the first part follows from normal cone characterisations of polyhedral sets \cite[Theorem 6.46]{rockafellar2009variational}. For part~\ref{item:rintnormal}, note that we can alternatively write $N_{\Delta^n}(y)$ as $D + \sum_{i:y_i=0} C_i$, where $D = \{(x,\dots,x)~|~x\in\mathbb{R}\}$ and $C_i = \{y\in\mathbb{R}^n\colon y_i \leq 0, y_j = 0~\forall j\neq i\}$. Then all sets are convex. Since $\relint D = D$ and $\relint C_i = \{y\in\mathbb{R}^n\colon y_i < 0, y_j = 0 \ \forall j\neq i\}$, the result follows by applying \cite[Corollary 6.6.2] {rockafellar1997convex}.
\end{proof}

    An mapping $F\colon U\to V$, between vector spaces $U$ and $V$, is said to be \textit{$L$-Lipschitz continuous} on a set $K\subseteq U$, for $L\in\mathbb{R}_+$, if $\|F(u) - F(v)\|\leq L\|u-v\|~\forall u,v\in K$. $F$ is \textit{locally} Lipschitz continuous  at a point $x_0\in U$ if there exists a neighbourhood $\mathcal{N}$ of $x_0$ and $L\in\mathbb{R}_+$, such that $F$ is  $L$-Lipschitz on $\mathcal{N}$, and locally Lipschitz on $K\subseteq U$ if it is locally Lipschitz at each $x_0\in K$. In the finite dimensional setting, this is equivalent to Lipschitz continuity on each compact subset $K^\prime\subseteq K$. If $F$ is (locally) Lipschitz on $\Hilbert$, then $F$ is simply said to be (locally) Lipschitz. 
    If $U=V=\Hilbert$, then $F$ is said to be \textit{monotone} if $\langle F(u)-F(v),u-v\rangle\geq 0~\forall u,v\in \Hilbert$.
    
    A function $\phi\colon\Hilbert_1\times\Hilbert_2\to\mathbb{R}$ is said to be \textit{convex-concave} if $\phi(\cdot,y)$ is convex for all $y\in \Hilbert_2$ and $\phi(x,\cdot)$ is concave for all $x\in\Hilbert_1$. A pair $(x^*,y^*)\in X\times Y\subseteq\Hilbert_1\times\Hilbert_2$ is said to be a \emph{saddle-point} of $\phi$ on $X\times Y$ if 
$$\phi(x^*,y)\leq\phi(x^*,y^*)\leq\phi(x,y^*)\quad\forall (x,y)\in X\times Y.$$
If $\phi$ is differentiable, the saddle operator $F\colon\Hilbert_1\times\Hilbert_2\to\Hilbert_1\times\Hilbert_2$ of $\phi$ is defined by
$$F(x,y) = (\nabla_x \phi(x,y), -\nabla_y \phi(x,y)).$$

\paragraph{KKT Theory.} Consider a general nonlinear program
\begin{equation}\label{eq:nlp}
\begin{array}{rrclcl}
\displaystyle \min_{x\in\Hilbert} & \multicolumn{1}{l}{ f(x)}\\
\textrm{s.t.} & g_i(x)\leq 0\quad\forall i=1,\dots,N,
\end{array}
\end{equation}
where $f\colon\Hilbert\to\mathbb{R},g_i\colon\Hilbert\to\mathbb{R}$ are assumed smooth and convex for all $i$. The \textit{Lagrangian function} associated with~\eqref{eq:nlp} is the function $\mathcal{L}\colon\Hilbert\times\mathbb{R}^N\to\mathbb{R},$ defined by
\begin{equation}\label{eq:lagrangian}
\mathcal{L}(x,\mu) := f(x) + \sum_{i=1}^N \mu_i g_i(x).
\end{equation}

The problem~\eqref{eq:nlp} satisfies the \emph{Mangasarian-Fromovitz Constraint Qualification (MFCQ)}~\cite{MANGASARIAN196737} at a feasible point $x_0\in\Hilbert$ if there exists $v\in\Hilbert$ such that
\begin{equation}\label{eq:MFCQ}
\langle\nabla g_i(x_0),v\rangle<0,
\end{equation}
for all indices $i\in\{1,\dots,N\}$ where $g_i(x_0)=0$. In the setting of~\eqref{eq:nlp}, this condition is equivalent to \textit{Slater's Condition} \cite[p.~45]{Borwein2006}. Under this assumption, $x^*\in\Hilbert$ is a solution to~\eqref{eq:minmax} if and only if there exists $\mu^*\in\mathbb{R}_+^N$ such that $(x^*,\mu^*)$ is a saddle point of $\mathcal{L}$ on $\Hilbert\times\mathbb{R}^N_+$ (see, for instance, \cite[Theorems 2.3.8/3.2.8 \& Proposition 3.2.3]{Borwein2006}). In this context, the components of $\mu^*$ are called \textit{KKT multipliers}. The MFCQ~\eqref{eq:MFCQ} is then known to be equivalent to boundedness of the set of KKT multipliers $\mu^*\in\mathbb{R}^N_+$ (see \cite{Gauvin1977ANA}). Furthermore, if we consider the index set $I^+$ of $i\in I$ such that a KKT multiplier $\mu^*_i>0$, and if in addition to~\eqref{eq:MFCQ} the set $\{\nabla g_i(x_0)~|~i\in I^+\}$ is linearly independent, then the problem~\eqref{eq:nlp} is said to satisfy the \textit{Strict Mangasarian Fromovitz Constraint Qualification (SMFCQ)}. This is equivalent to uniqueness of the KKT multipliers (see \cite[Proposition 1.1]{kyparisis1985uniqueness}) at $x_0$.

The saddle point $(x^*,\mu^*)$ described above satisfies the following conditions (called \textit{KKT condtions})

\begin{itemize}
    \item (Stationarity) $0 = \nabla_x \mathcal{L}(x^*,\mu^*)=\nabla f(x^*)+\sum_{i=1}^N\mu^*_i\nabla g_i(x^*)$.
    \item (Feasibility) $g_i(x^*)\leq 0$ for all $i=1,\dots,N$.
    \item (Complementary Slackness) $\mu^*_i g_i(x^*) = 0$ for all $i=1,\dots,N$.
\end{itemize}

In addition, the pair $(x^*,\mu^*)$ is said to satisfy the \textit{strict complementary slackness condition} if $\mu^*_i=0\iff g_i(x^*)>0$ for all $i$.

\section{Saddle Reformulation}\label{sec:saddle}

In this section, we justify reformulating the finite-max minimisation problem~\eqref{eq:minmax} as the saddle point problem~\eqref{eq:saddle}. In particular, we explore the correspondence between minimisers $f$ and saddle points of $\phi$. We then discuss some properties of saddle-formulation, and show how it can be solved within the \textit{Variational Inequality} framework. In what follows, the active support at a point $x\in\Hilbert$ is denoted
$$I(x) := \{i\in I\colon f_i(x)=f(x)\}.$$
The following result then explores the correspondence between \eqref{eq:minmax} and \eqref{eq:saddle}.
\begin{theorem}\label{thm:reform}Let $f=\max\{f_1,\dots,f_N\}$ be as defined in~\eqref{eq:minmax}, and $\phi$ as in~\eqref{eq:saddle}. Then the following conditions are equivalent.
    \begin{enumerate}[(a)]
        \item $x^*\in\argmin f$
        \item\label{item:opt conditions} There exists $y^*\in\Delta^N$ such that $\nabla_x \phi(x^*,y^*)=0$ and $f(x^*)=\phi(x^*,y^*)$
        \item There exists $y^*\in\Delta^N$ such that $(x^*,y^*)$ is a saddle point on $\phi$ on $\Hilbert\times\Delta^N$.
    \end{enumerate}
    Moreover, for $y^*\in\Delta^N$ in (b) and (c), $y^*_i>0$ only if $i\in I(x^*)$.
\end{theorem}
\begin{proof}$(a)\implies(b)$ As a result of Lemma~\ref{lem:subdiff}, there exists $y^*\in\Delta^N$ such that $y^*_i=0$ for all $i\notin I(x^*)$ and $0 = \sum_{i\in I}y^*_i\nabla f_i(x^*)=\nabla_x \phi(x^*,y^*)$. Then, since $f_i(x^*)=f(x^*)$ for all $i\in I(x^*)$ by definition, it follows that 
$$\phi(x^*,y^*) = \sum_{i\in I} y^*_i f_i(x^*) = \sum_{i\in I(x^*)}y^*_i f(x^*) = f(x^*).$$
    
$(b)\implies(c)$ Since $0 = \nabla_x\phi(x^*,y^*)$ and $\phi(\cdot,y^*)$ is convex, it follows that $\phi(x^*,y^*)\leq \phi(x,y^*)$ for all $x\in\Hilbert$. For $y\in\Delta^N$, we have
$$\phi(x^*,y) = \sum_{i\in I} y_i f_i(x^*) \leq \sum_{i\in I}y_i f(x^*) = f(x^*) =\phi(x^*,y^*),$$ 
The result follows by combing these two inequalities.

$(c)\implies(a)$ Let $x\in\Hilbert$ be arbitrary, and note that $f(x) = \max_{y\in\Delta^N} \phi(x,y)$. Then since $(x^*,y^*)$ is a saddle point of $\phi$ on $\Hilbert\times\Delta^N$, we see that
$$f(x^*) = \max_{y\in\Delta^N}\phi(x^*,y) = \phi(x^*,y^*)\leq\phi(x,y^*)\leq\max_{y\in\Delta^N}\phi(x,y)=f(x),$$
which concludes the proof.
\end{proof}


Theorem~\ref{thm:reform} shows that every solution $x^*$ of the nonsmooth minimisation problem~\eqref{eq:minmax} corresponds to a saddle-point $(x^*,y^*)$ of \eqref{eq:saddle} for some $y^*\in\Delta^N$, and vice versa. In the latter setting, it also shows that the variable $y^*$ contains additional information about $f$ at $x^*$. In particular, the support of $y^*$ ($\{i\in I\colon y^*_i>0\}$) is a subset of $I(x^*)$, and $y^*$ consists of convex multipliers satisfying $0=\sum_{i\in I(x^*)}y_i^*\nabla f_i(x^*)\in\partial f(x^*)$.

Given a minimiser $x^*$ of $f$, we emphasise that $\overline{y}\in\max_{y\in\Delta^N}\phi(x^*,y)$ does not necessarily imply that $(x^*,\overline{y})$ is a saddle-point of $\phi$ on $\Hilbert\times\Delta^N$, even though $f(x^*)=\phi(x^*,\overline{y})=\phi(x^*,y^*)$. This is illustrated in the following example.
\begin{example}\label{ex:saddle}
Consider $f=\max\{f_1,f_2\}$ where the subfunctions $f_1,f_2\colon\mathbb{R}\to\mathbb{R}$ are given by
$$f_1(x) = (x+1)^2, \quad f_2(x) = (x-1)^2.$$
Then $x^*=0$ is the unique minimiser of $f$, and $\overline{y}=(1,0)\in\argmax_{y\in\Delta^2}\phi(x^*,y)$. Although $f(x^*)=\phi(x^*,\overline{y})$,  the point $(x^*,\overline{y})$ is not saddle point of $\phi$ on $\Hilbert\times\Delta^2$ as
$$\phi(x^*,\overline{y}) = (x^*+1)^2 = 1 > 0 = \phi(-1,\overline{y}). $$
However, by revisiting Theorem~\ref{thm:reform}\ref{item:opt conditions}, it can be verified that the unique saddle-point of $\phi$ on $\Hilbert\times\Delta^2$ is given by $(x^*,y^*)$ with $y^*=\left(\frac{1}{2}, \frac{1}{2}\right)$.
\end{example}

For fixed $x^*\in\argmin f$, we denote by $\Lambda(x^*)$ the corresponding solutions $y^*\in\Delta^N$ such that $(x^*,y^*)$ satisfy the conditions of Theorem \ref{thm:reform}. That is,
$$\Lambda(x^*) := \{y^*\in\Delta^N\colon (x^*,y^*)\text{~is a saddle point of~}\phi\text{~on~}\Hilbert\times\Delta^N\}.$$

We will now consider two refinements of Theorem~\ref{thm:reform}, which will be of use in the subsequent sections. We will distinguish between the active support $I(x^*)$ 
and the \textit{strongly} active support 
$$I^+(x^*):=\{i\in I\colon\exists~y\in\Lambda(x^*)\text{~s.t.~}y_i>0\}.$$
Then $I^+(x^*)\subseteq I(x^*)$ as discussed in Theorem~\ref{thm:reform}, with equality if and only if the nondegeneracy condition holds in the following theorem. Where $I(x^*)$ denotes the set of active subfunctions, the set $I^+(x^*)$ asks for a combination gradients such that $0$ is contained in the convex hull.

\begin{theorem}\label{thm:nondegeneracy}
    Let $x^*\in\argmin f$, where $f=\max\{f_1,\dots,f_N\}$ as in~\eqref{eq:minmax}. Then the following conditions are equivalent.
    \begin{enumerate}[label=(\alph*)]
        \item\label{item:nondegen} $0\in\relint\partial f(x^*)$.
        \item There exists $y^*\in\Lambda(x^*)$ such that $\left(f_1(x^*),\dots,f_N(x^*)\right)\in\relint N_{\Delta^N}(y^*)$.
        \item $I(x^*) = I^+(x^*)$.
    \end{enumerate}
\end{theorem}
\begin{proof}
$(a)\implies(b)$ Lemmas~\ref{lem:subdiff}\ref{item:max-rule} and~\ref{lem:rintCH} together imply the existence of $y^*\in\Delta^N$ such that $y^*_i>0$ if and only if $i\in I(x^*)$, and 
$$\sum_{i\in I(x^*)}y^*_i\nabla f_i(x^*) = 0.$$
Then $\sum_{i\in I}y^*_i f_i(x^*) = \sum_{i\in I(x^*)} y^*_i f(x^*) = f(x^*)$, and it follows from Theorem~\ref{thm:reform}\ref{item:opt conditions} that $y^*\in\Lambda(x^*)$. Then by applying Lemma~\ref{lem:cone}\ref{item:normal} we see that
$$(f_1(x^*),\dots,f_N(x^*)) = f(x^*) e - (f(x^*) - f_1(x^*),\dots,f(x^*) - f_N(x^*))\in\relint N_{\Delta^N}(y^*),$$
where the inclusion follows since $w_i=f(x^*) - f_i(x^*) = 0$ if and only if $y^*_i>0$. $(b)\implies(c)$ Note that the one sided inclusion $I^+(x^*)\subseteq I(x^*)$ already holds. Then if $\left(f_1(x^*),\dots,f_N(x^*)\right)\in\relint N_{\Delta^N}(y^*)$, from Lemma~\ref{lem:cone}\ref{item:rintnormal} we have $w_i = f(x^*) - f_i(x^*) = 0$ if and only if $y^*_i>0$, and hence
$$I(x^*) = \{i\in I\colon y^*_i>0\}\subseteq I^+(x^*).$$
$(c)\implies(a)$ From Lemma~\ref{lem:rintCH} and Theorem~\ref{thm:reform}, we observe that
    $$0 = \sum_{i\in I^+(x^*)} y^*_i\nabla f_i(x^*) = \sum_{i\in I(x^*)}y^*_i \nabla f_i(x^*)\in\relint\partial f(x^*).$$
This completes the proof.    
\end{proof}
Note that~\ref{item:nondegen} in Theorem~\ref{thm:nondegeneracy} is called the \textit{non-degeneracy} condition and is a common assumption in the literature of constraint/support identification. 

\begin{theorem}\label{thm:SMFCQ}
Let $x^*\in\argmin f$, where $f=\max\{f_1,\dots,f_N\}$ as in~\eqref{eq:minmax}. Then the vectors of the set $\{\nabla f_i(x^*)~|~i\in I^+(x^*)\}$ are affinely independent if and only if $\Lambda(x^*)$ is a singleton. 
\end{theorem}
\begin{proof}We first observe as a result of Theorem~\ref{thm:reform} that $\Lambda(x^*)$ is non-empty. By rearranging if necessary, we can suppose without loss of generality that $I^+(x^*)=\{1,\dots,n\}$ for some $n\leq N$.

$(\implies)$ Let $y^*,\overline{y}\in\Lambda(x^*)$. Then
$$0 = \sum_{i=1}^n y^*_i\nabla f_i(x^*) - \sum_{i=1}^n \overline{y}_i\nabla f_i(x^*) = \sum_{i=1}^n (y^*_i - \overline{y}_i)\nabla f_i(x^*),$$
and $\sum_{i=1}^n y^*_i - \overline{y}_i = 1-1 = 0$. Since $\{\nabla f_i(x^*)~|~i\in I^+(x^*)\}$ is affinely independent, we have $y^*_i - \overline{y}_i = 0$ for all $i\in I^+(x^*)$, thus $y^*=\overline{y}$ and we conclude that $\Lambda(x^*)$ is a singleton.

$(\impliedby)$ Let $y^*\in\Lambda(x^*)$ and suppose $\{\nabla f_i(x^*)~|~i\in I^+(x^*)\}$ is affinely dependent. Then there exist $\alpha=(\alpha_1,\dots,\alpha_n,0,\dots,0)\in\mathbb{R}^N$ and an index $j\in I^+(x^*)$ such that 
\begin{equation}\label{eq:alpha}
\alpha_j\neq 0,\quad 0 =\sum_{i=1}^n\alpha_i\nabla f_i(x^*),\quad \sum_{i=1}^n\alpha_i = 0.
\end{equation}
Since $y^*_i>0$ for all $i\in I^+(x^*)$, we can scale $\alpha$ without loss of generality (while preserving \eqref{eq:alpha}) so that $y^*_i+\alpha_i\geq 0$ for all $i\in I^+$. Then $\sum_{i=1}^n y^*_i+\alpha_i = 1$, and $0 = \sum_{i=1}^n(y^*_i+\alpha_i)\nabla f_i(x^*) = 0$. It follows that $y^*+\alpha\in\Lambda(x^*)$, but $y^*\neq y^*+\alpha$ since  $\alpha_j\neq 0$. Thus, $\Lambda(x^*)$ is not singleton, which completes the proof.
\end{proof}

It is important to note that Theorem~\ref{thm:SMFCQ} does not ensure uniqueness of the maximisers $\overline{y}$ of $\phi(x^*,\cdot)$, and that $y^*\in\Lambda(x^*)$ is characterised by the additional requirement of $0 = \sum_{i\in I} y^*_i \nabla f_i(x^*)$ as discussed in Example \ref{ex:saddle}. We will visualise the preceding two theorems for a simple example in Figure~\ref{fig:examples}.

\begin{figure}[h!]
    \centering
     \begin{subfigure}[t]{.325\linewidth}
    \centering\includegraphics[width=\linewidth]{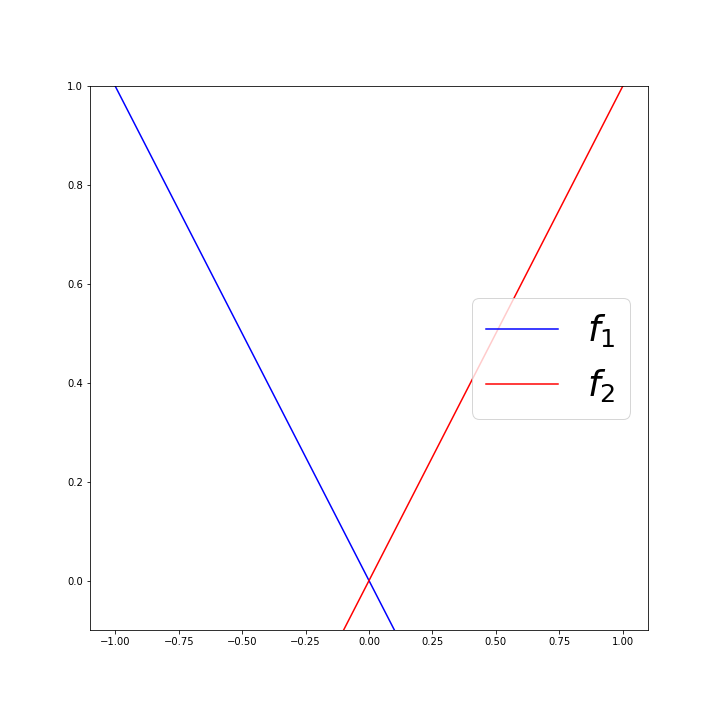}
    \caption{$f_1(x) = -x, f_2(x) = x$}
  \end{subfigure}
  \begin{subfigure}[t]{.325\linewidth}
    \centering\includegraphics[width=\linewidth]{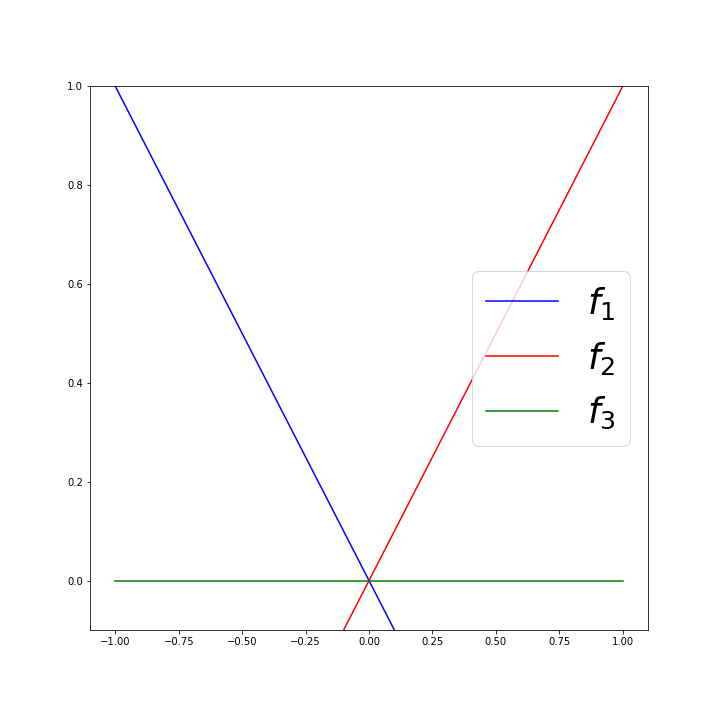}
    \caption{ $f_1(x) = -x, f_2(x) = x,$\\ $f_3(x) = 0$}
  \end{subfigure}
  \begin{subfigure}[t]{.325\linewidth}
    \centering\includegraphics[width=\linewidth]{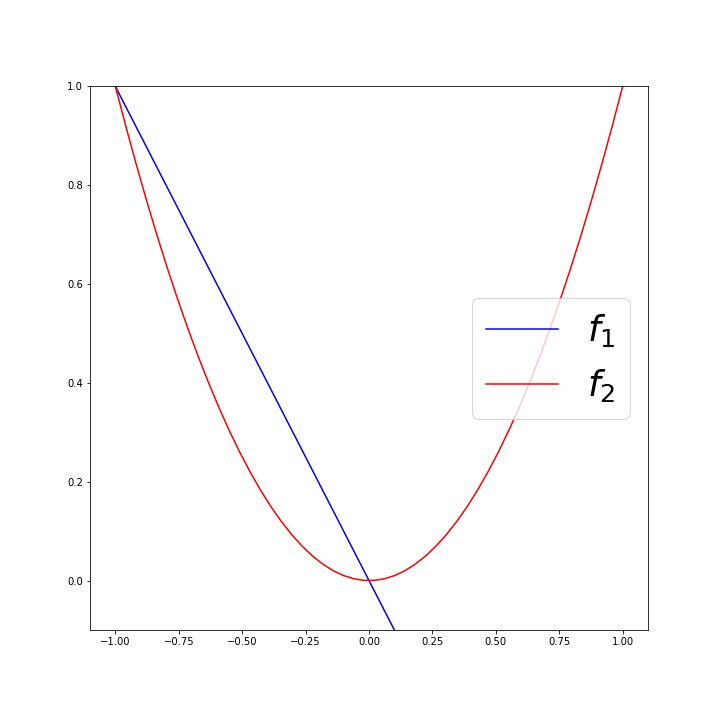}
    \caption{$f_1(x) = -x, f_2(x) = x^2$}
  \end{subfigure}
  \caption{Examples satisfying both Theorems~\ref{thm:nondegeneracy} and~\ref{thm:SMFCQ} (a), only Theorem~\ref{thm:nondegeneracy} (b), and only Theorem~\ref{thm:SMFCQ} (c).}
  \label{fig:examples}
\end{figure}

In all cases, $x^*=0$ and $I(x^*) = I$, so points in $\overline{y}\in\argmax_{y\in\Delta^N} \phi(x^*,y)$ can be chosen as any $\overline{y}\in\Delta^N$, as in Example~\ref{ex:saddle}. However, solving for $0 = y^*_1\nabla f_1(x^*) + y^*_2\nabla f_2(x^*) = -y^*_1 + y^*_2$, together with $y^*\in\Delta^2$, shows that $y^*=(\frac{1}{2},\frac{1}{2})\in\Lambda(x^*)$ is unique in part (a). Then $I^+(x^*)=\{1,2\}=I(x^*)$ and (a) satisfies both Theorems. Part (b) is similar, except $\Lambda(x^*)=\{y^*\in\Delta^3\colon y^*_1=y^*_2\}$ is not a singleton, but we still have nondegeneracy since $(\frac{1}{2},\frac{1}{2},0),(0,0,1)\in\Lambda(x^*)$ and thus $I^+(x^*)=\{1,2,3\}=I(x^*)$. Finally, in part (c) we have uniqueness of $y^*=(0,1)\in\Lambda(x^*)$, but $I^+(x^*)=\{2\}\subset I(x^*)=\{1,2\}$. Therefore, (c) satisfies Theorem~\ref{thm:SMFCQ} but not nondegeneracy.

\begin{remark}
    By introducing an auxiliary variable $\tau\in\mathbb{R}$ and considering the \textit{epigraphical formulation} of~\eqref{eq:minmax} given by
    \begin{equation}\label{eq:epi}
    \begin{aligned}
        \min_{(x,\tau)\in\Hilbert\times\mathbb{R}} & \tau\\
        \text{s.t~} & f_i(x) - \tau\leq 0\quad\forall i=1,\dots,N,
    \end{aligned}
    \end{equation}
    and the corresponding Lagrangian 
    $$\mathcal{L}((x,\tau),\mu) = \tau + \sum_{i=1}^N\mu_i(f_i(x)-\tau),$$
    we observe that $y^*\in\Lambda(x^*)$ can be equivalently defined as the KKT vector for this system. This largely follows from the stationarity condition $0=\frac{\partial\mathcal{L}((x^*,\tau^*),\mu^*)}{\partial\tau} = 1-\sum_{i=1}^N\mu^*_i$. In this context, the Theorems of this section can be analoguously written in terms of the KKT conditions. Theorem~\ref{thm:reform}\ref{item:opt conditions} is equivalent to stationarity and complementary slackness. Theorem~\ref{thm:nondegeneracy} is the same as strict complementary slackness. Finally, Theorem~\ref{thm:SMFCQ} amounts to the SMFCQ, noting that the standard MFCQ automatically applies to~\eqref{eq:epi}, and the SMFCQ is known to be equivalent to the uniqueness of the KKT multipliers \cite[Proposition 1.1]{kyparisis1985uniqueness}.
\end{remark}

The point $(x^*,y^*)$ is a saddle-point of $\phi$ if and only if it satisfies the following pair of equations, noting that $\partial(-\phi(x^*,\cdot) + \iota_{\Delta^N}) = -\nabla_y\phi(x^*,\cdot) + N_{\Delta^N}$ as a result of the subdifferential sum rule \cite[Corollary 16.48]{Bauschke2017} since $\dom\phi(x^*,\cdot)=\mathbb{R}^N$.
\begin{equation}\label{eq:first-order saddle}
    \left\{\begin{aligned}
    \nabla_x\phi(x^*,y^*) &= 0\\
    \langle-\nabla_y\phi(x^*,y^*),y-y^*\rangle &\geq 0\quad\forall y\in\Delta^N.
\end{aligned}\right.
\end{equation}
The system \eqref{eq:first-order saddle} can be compactly expressed as the \emph{variational inequality}
\begin{equation}\label{eq:VIP}
    \langle F(z^*),z-z^*\rangle\geq 0\quad\forall z\in K,
\end{equation}
where
$ z=(x,y)\in\Hilbert\oplus\mathbb{R}^N$, $F(z) = (\nabla_x\phi(x,y),-\nabla_y\phi(x,y))$, and $K=\Hilbert\times\Delta^N.$ 

We will conclude this section by summarising properties of the variational inequality \eqref{eq:first-order saddle} that will be useful for applying convergence results of  iterative algorithms (\emph{e.g.,} \cite{facchinei2003finite, kinderlehrer2000introduction, malitsky2020golden}) for solving~\eqref{eq:VIP}.
\begin{proposition}\label{prop:prop}Let $f=\max\{f_1,\dots,f_N\}$ be as defined in~\eqref{eq:minmax}, and $\phi$ as in~\eqref{eq:saddle}. Consider the saddle operator given by $F(z) = (\nabla_x\phi(x,y),-\nabla_y\phi(x,y))$ for $z=(x,y)\in\Hilbert\times\Delta^N$. Then the following assertions hold.
    \begin{enumerate}[label=(\alph*)]
        \item $F$ is monotone.
        \item If $\nabla f_i$ is $L$-Lipschitz continuous in a set $K\subseteq\Hilbert$ and $\|\nabla f_i(x)\|\leq M$ for all $x\in K$ and $i\in I$, then $F$ is $\overline{L}$-Lipschitz continuous on $K\times\Delta^N$ with $\overline{L}=\sqrt{L^2 + 2M^2N}$.
        \item If $\nabla f_i$ is locally Lipschitz continuous for all $i\in I$, then $F$ is locally Lipschitz continuous.
    \end{enumerate}
\end{proposition}
\begin{proof}
    (a) Since $\phi$ is convex-concave and finite everywhere, monotonicity of $F$ follows from~\cite[Corollary 2]{rockafellar1970monotone}.
    (b) We first note that $\|\nabla f_i(x)\|\leq M$ for all $x\in K$ implies that $f_i$ is $M$-Lipschitz \cite[Theorem 9.7]{rockafellar2009variational} on $K$. Let $x,x^\prime\in\Hilbert$ and $y,y^\prime\in\Delta^N$. Then, we have
        \begin{equation}\label{eq:4ineq}\begin{aligned}
        \|\nabla_x\phi(x,y) - \nabla_x\phi(x,y^\prime)\| 
        &= \left\|\sum_{i\in I}(y_i-y^\prime_i)\nabla f_i(x)\right\|\\
        &\leq \sum_{i\in I}\lvert y_i - y^\prime_i\rvert\|\nabla f_i(x)\|
        \leq M\|y-y^\prime\|_1\leq M\sqrt{N}\|y-y^\prime\|,\\
        \|\nabla_x\phi(x,y^\prime) - \nabla_x\phi(x^\prime,y^\prime)\| 
        &= \left\|\sum_{i\in I}y^\prime_i(\nabla f_i(x) - \nabla f_i(x^\prime))\right\|\\
        &\leq\sum_{i\in I}y^\prime_i\|\nabla f_i(x) - \nabla f_i(x^\prime)\|\leq L\|x-x^\prime\|,\\
        \|\nabla_y\phi(x,y) - \nabla_y\phi(x,y^\prime)\| &= \|(f_i(x)-f_i(x))_{i\in I}\| = 0,\\
        \|\nabla_y\phi(x,y^\prime) - \nabla_y\phi(x^\prime,y^\prime)\| &= \|(f_i(x)-f_i(x^\prime))_{i\in I}\| \leq M\sqrt{N}\|x-x^\prime\|.
    \end{aligned}\end{equation}
    Denote $z=(x,y),z^\prime=(x^\prime,y^\prime)\in\Hilbert\times\Delta^N$. Then expanding the norm-squared and applying Young's inequality followed by \eqref{eq:4ineq} gives
    \begin{align*}
       \|F(z) - F(z^\prime)\|^2 
        &= \|\nabla_x\phi(x,y) - \nabla_x\phi(x^\prime,y^\prime)\|^2 + \|\nabla_y\phi(x,y) - \nabla_y\phi(x^\prime,y^\prime)\|^2\\
        &\leq 2\|\nabla_x\phi(x,y) - \nabla_x\phi(x,y^\prime)\|^2 + 2\|\nabla_x\phi(x,y^\prime) - \nabla_x\phi(x^\prime,y^\prime)\|^2\\ 
        &\quad + \|\nabla_y\phi(x,y) - \nabla_y\phi(x,y^\prime) + \nabla_y\phi(x,y^\prime) - \nabla_y\phi(x^\prime,y^\prime)\|^2\\ 
        &\leq L^2\|x-x^\prime\|^2 + 2M^2 N(\|x-x^\prime\|^2 + \|y-y^\prime\|^2) \leq \overline{L}^2\|z-z^\prime\|^2. 
    \end{align*}
    (c) Follows after applying (b) for any compact $K$, noting that $\nabla f_i$ is Lipschitz on such sets by assumption and locally bounded from continuity.
\end{proof}

In the setting of Proposition~\ref{prop:prop}, (global) Lipschitz continuity of the saddle operator $F$ is verified when both $f_i$ and $\nabla f_i$ are Lipschitz continuous for all $i\in I$. This is too restrictive to be applicable in many important settings. Moreover, Proposition~\ref{prop:prop} gives some information on this local curvature. In particular, we observe the local Lipschitz constant is more dependent on to the local Lipschitz/boundedness of $\nabla f_i$, than on the number $N$ of subfunctions.

\section{Active Set Identification}\label{sec:identification}

In this section, we collect and establish results on active support identification for the nonsmooth minimisation problem~\eqref{eq:minmax} and the smooth min-max formulation~\eqref{eq:saddle}. We will make use of the following assumptions.






\begin{assumption}\label{assumptions} In the context of~\eqref{eq:minmax}, consider the following properties.
\begin{enumerate}[label=A.\arabic*]
    \item\label{ass:nondegen} $x^*\in\Hilbert$ is a non-degenerate minimiser of $f$, that is, $0\in\relint\partial f(x^*)$.
    \item\label{ass:SMFCQ} $\Lambda(x^*)$ is a singleton.
\end{enumerate}
\end{assumption}

Throughout this section, we consider a pair of sequences  $((x^k,y^k))\subset\Hilbert\times\Delta^N$ converging to a saddle point $(x^*,y^*)$ of $\phi$ on $\Hilbert\times \Delta^N$.

\subsection{Direct Identification via the Saddle Reformulation}

\begin{theorem}\label{thm:subgradient}
Consider $f=\max\{f_1,\dots,f_N\}$ as in~\eqref{eq:minmax}, and let $(x^k)$ be a sequence converging to $x^*\in\argmin f$. If there exists $k_0\in\mathbb{N}$ such that
    $$I(x^k) = I(x^*)\text{~for all $k\geq k_0$},$$
    then $\dist(0,\partial f(x^k))\to 0$. Moreover, if \ref{ass:nondegen} and \ref{ass:SMFCQ} hold, then the reverse implication also holds.
\end{theorem}
\begin{proof}
$(\implies)$ Consider $y^*\in\Lambda(x^*)$, so that $0=\sum_{i\in I(x^*)}y^*_i\nabla f_i(x^*)$ from Theorem~\ref{thm:reform}\ref{item:opt conditions}, and consider the sequence given by
$\xi^k := \sum_{i\in I(x^*)}y^*_i\nabla f_i(x^k)$. From continuity of $\nabla f_i$, it follows that $\xi^k\to 0$ as $k\to\infty$. Let $k\geq k_0$, so that $I(x^*)=I(x^k)$. The result is given by applying the max rule found in Lemma~\ref{lem:subdiff}\ref{item:max-rule} as follows
$$\xi^k\in\conv\{\nabla f_i(x^k)~|~i\in I(x^*)\} = \conv\{\nabla f_i(x^k)~|~i\in I(x^k)\}=\partial f(x^k),$$
which implies $\dist(0,\partial f(x^k))\leq \|\xi^k\|\to 0$, thus completing the forward direction.

$(\impliedby)$ Since $\range I(\cdot)$ is finite, there exists an index set $I^\prime\subseteq I$ such that $I(x^k)=I^\prime$ for infinitely many $k\in\mathbb{N}$. So there is a subsequence $(x^{k_j})_{j\in\mathbb{N}}$ such that $I(x^{k_j}) = I^\prime$ for all $j\in \mathbb{N}$.

Let $r\in I^\prime$ be arbitrary. Then $f_r(x^{k_j}) = f(x^{k_j})$ for all $j\in \mathbb{N}$, and hence also in the limit $f_r(x^*) = f(x^*)$ from continuity of $f_r$ and $f$. This shows that $r\in I(x^*)$, and therefore  $I^\prime\subseteq I(x^*)$. 

Now, since $\dist(0,\partial f(x^{k_j}))\to 0$, there exists a sequence of subgradients $\xi^{k_j}\in\partial f(x^{k_j})$ such that $\|\xi^{k_j}\|\to 0$. Also, by Lemma~\ref{lem:subdiff}\ref{item:max-rule}, there exists a corresponding sequence $y^{k_j}\in\Delta^N$ such that $\xi^{k_j}=\sum_{i\in I^\prime}y^{k_j}\nabla f_i(x^{k_j})$. Since $(y^{k_j})$ is bounded, we can suppose without loss of generality that $y^{k_j}$ converges to some $y^*\in\Delta^N$ as $j\to\infty$ (otherwise choose a new subsequence). Note that, for $i\notin I^\prime$, we have $y^*_i=\lim_{j\to\infty}y^{k_j}_i=0$. Therefore
$$0=\lim_{j\to\infty}\xi^{k_j}=\lim_{j\to\infty}\sum_{i\in I^\prime}y^{k_j}_i\nabla f_i(x^{k_j}) = \sum_{i\in I^\prime} y^*_i\nabla f_i(x^*)=\sum_{i\in I}y^*_i\nabla f_i(x^*).$$
Also, note that $f(x^*)=f_i(x^*)$ for all $i\in I^\prime$, then
$$\sum_{i\in I^\prime}y^*_i f_i(x^*) = \sum_{i\in I^\prime}y^*_i f(x^*) = f(x^*).$$
Therefore, it follows from Theorem~\ref{thm:reform}\ref{item:opt conditions} that $y^*\in\Lambda(x^*)$.
Then, since $\Lambda(x^*)=\{y^*\}$ under~\ref{ass:SMFCQ}, we have 
$$I^+(x^*)=\{i\in I\colon \exists~y\in\Lambda(x^*)\text{~s.t.}~y_i>0\}=\{i\in I\colon y^*_i>0\}\subseteq I^\prime.$$
By invoking Theorem~\ref{thm:nondegeneracy} under~\ref{ass:nondegen}, we then conclude that $I(x^*)=I^+(x^*)\subseteq I^\prime$. 

Altogether, we have established that $I^\prime=I(x^*)$. Since $I^\prime$ was chosen arbitrarily (as an index set satisfying $I(x^k)=I^\prime$ for infinitely many $k\in\mathbb{N})$, there exists $k_0\in\mathbb{N}$ such that $I(x^k)=I(x^*)$ for all $k\geq k_0$. This completes the proof.


\end{proof}

\begin{remark}
While we can achieve some level of support identification without either condition in Assumption~\ref{assumptions}, the result shown in Theorem~\ref{thm:subgradient} only holds under both of these conditions. In the absence of both~\ref{ass:SMFCQ} and~\ref{ass:nondegen}, we can identify at least one subset of indicies $I^\prime\subseteq I(x^*)$ such that $0\in\conv\{\nabla f_i(x^*)~|~i\in I^\prime\}$. If~\ref{ass:SMFCQ} holds but not~\ref{ass:nondegen}, then we can derive
$$I^+(x^*)\subseteq I(x^k)\subseteq I(x^*)\text{~for all $k\geq k_0$}.$$
In either case, the obtained result would still be useful for simplifying the problem as desired, though we have decided to keep all assumptions for simplicity.
\end{remark}
On its own, Theorem~\ref{thm:subgradient} is generally not applicable in the context of first-order methods, since it is difficult to ensure whether a sequence $(x^k)$ satisfies the necessary and sufficient condition $\dist(0,\partial f(x^k))\to 0$. This is due to the continuity properties of $\partial f$. Indeed, $\partial f$ is \textit{outer semicontinuous} \cite[Theorem 24.4]{rockafellar1997convex} at $x^*$, that is,
$$x^k\to x^*,~\xi^k\in\partial f(x^k)\implies \xi^k\to \xi^*\in\partial f(x^*),$$
but can fail to be \textit{inner semi continuous} at $x^*$, that is,
$$\forall\xi^*\in\partial f(x^*),~\exists x^k\to x^*,~\xi^k\in\partial f(x^k)\text{~s.t.~}\xi^k\to\xi^*.$$
Therefore it is not guaranteed that $\dist(0,\partial f(x^k))\to 0$ as $k\to\infty$.

Instead, one could consider a ``shadow sequence'' of $(x^k)$ for which this condition does apply, one basic example being the proximal operator $(\prox_{\lambda f}(x^k))$ for some fixed $\lambda>0$. Since $\frac{1}{\lambda}(x^k-\prox_{\lambda f}(x^k)) \in \partial f(\prox_{\lambda f}(x^k))$, we have $$\dist(0, \partial f(\prox_{\lambda f}(x^k)))\leq\frac{1}{\lambda}\|x^k-\prox_{\lambda f}(x^k)\|\to 0.$$ However, the class of functions $f$ for which $\prox_{\lambda f}$ is available is not very interesting, since we could simply minimise $f$ directly.

The $\varepsilon$-subdifferential, unlike the standard differential operator, is inner semicontinuous (see \cite[Section XI.4.1]{Hiriart-Urruty1993}) in the following sense

$$\forall \xi^*\in\partial f(x^*),~\exists x^k\to x^*,~\varepsilon_k\to 0,~\xi^k_{\varepsilon}\in\partial_{\varepsilon_k}f(x^k)\quad\text{s.t.~}\xi^k_{\varepsilon}\to\xi^*.$$

The following results will therefore describe how we can instead use the use the $\varepsilon$-subdifferential to approximate $I(x^*)$.

\begin{theorem}\label{thm:eps-subdiff}
Consider $f=\max\{f_1,\dots,f_N\}$ as in~\eqref{eq:minmax}, and let $(x^k)$ be a sequence converging to $x^*\in\argmin f$. Suppose \ref{ass:nondegen} and \ref{ass:SMFCQ} hold. If $\dist(0,\partial_{\varepsilon_k} f(x^k))\to 0$ for some positive sequence $\varepsilon_k\to 0$, and $p>1$, then there exists $k_0\in\mathbb{N}$ such that
    $$I_{\text{eps}}(x^k) := \left\{i\in I\colon f(x^k)-f_i(x^k)\leq \varepsilon_k^{(p-1)/p}\right\}=I(x^*)\text{~for all~} k\geq k_0$$
\end{theorem}
\begin{proof}
     First, we note that since $f_i$ is convex, it is locally Lipschitz (see \cite[Theorem 4.1.3]{Borwein2006}). In particular, $f_i$ is $L$-Lipschitz on a neighbourhood $\mathcal{N}$ of $x^*$. 
     It follows that $f$ is also $L$-Lipschitz (see \cite[Proposition 9.10]{rockafellar2009variational}) on $\mathcal{N}$.

     Now, to show the $(\subseteq)$ inclusion, take $r\notin I(x^*)$. Then $f(x^*)-f_r(x^*)>0$ and $\varepsilon_k^{(p-1)/p}\to 0$, so $f(x^k)-f_r(x^k)>\varepsilon_k^{(p-1)/p}$ for $k$ sufficiently large by continuity, which implies $r\not\in I_{\rm eps}(x^*)$. From this, we conclude that $I_{\rm eps}(x^k)\subseteq I(x^*)$ for all sufficiently large $k$.

     For the $(\supseteq)$ direction we first take $\xi^k\in\partial_{\varepsilon_k} f(x^k)$ such that $\|\xi^k\|\to 0$ as $k\to\infty$. We then deduce from Theorem~\ref{thm:BR} that for any $\delta_k>0$, there exists $u^k\in\Hilbert$ and $\nu^k\in\partial f(u^k)$ such that
    $$\|u^k-x^k\|\leq\frac{\varepsilon_k}{\delta_k}, \text{~and~} \|\nu^k-\xi^k\|\leq \delta_k,$$
    for all $k\in\mathbb{N}$. Let $\delta_k=2L\varepsilon_k^{1/p}\to 0$, so that $\frac{\varepsilon_k}{\delta_k}=\frac{1}{2L}\varepsilon_k^{(p-1)/p}\to 0$ as $k\to\infty$. Then, since $\|\nu^k-\xi^k\|\leq\delta_k\to 0$, it follows that $\dist(0,\partial f(u^k))\leq\|\nu^k\|\to 0$ as $k\to\infty$. Therefore by applying Theorem~\ref{thm:subgradient}, which holds under \ref{ass:nondegen} and \ref{ass:SMFCQ}, to $(u^k)$, we deduce the existence of $k_0\in\mathbb{N}$ such that $I(u^k)=I(x^*)$ for all $k\geq k_0$. Let $i\in I(x^*)$ be arbitrary, so that $f_i(u^k) = f(u^k)$ for all $k\geq k_0$. Note that $(u^k)$ and $(x^k)$ both converge to $x^*$, so by increasing $k_0$ if necessary, we have $u^k,x^k\in\mathcal{N}$ for $k\geq k_0$. Then by applying the Lipschitz inequality, we observe that
    \begin{align*}
        f(x^k)-f_i(x^k) &= f(x^k)-f(u^k)+f_i(u^k)-f_i(x^k) \\
        &\leq\lvert f(x^k)-f(u^k)\rvert+\lvert f_i(u^k)-f_i(x^k)\rvert\\
        &\leq 2L\|x^k-u^k\|\leq 2L\frac{\varepsilon_k}{\delta_k} = \varepsilon^{(p-1)/p}.
    \end{align*}
    Therefore $i\in I_{\rm eps}(x^k,y^k)$, which shows that $I(x^*)\subseteq I_{\rm eps}(x^k,y^k)$ for $k\geq k_0$, which completes the proof.
\end{proof}

In order to apply Theorem~\ref{thm:eps-subdiff}, knowledge of the sequence $(\varepsilon_k)$ such that ${\dist(0,\partial_{\varepsilon_k}f(x^k))\to0}$ is required. In the context of the saddle reformulation \eqref{eq:saddle}, this is given by the following result.


\begin{corollary}\label{cor:eps-subdiff}
Consider $f=\max\{f_1,\dots,f_N\}$ and $\phi$ as in~\eqref{eq:minmax} and~\eqref{eq:saddle}, respectively. Let $\bigl((x^k,y^k)\bigr)\subseteq\Hilbert\times\Delta^N$ be a sequence converging to a saddle point $(x^*,y^*)$ of $\phi$ on $\Hilbert\times\Delta^N$, and let $(\varepsilon_k)$ be the sequence given by $\varepsilon_k:=f(x^k)-\phi(x^k,y^k)\geq 0$. Then $\varepsilon_k\to 0$ and
\begin{equation}\label{eq:e-subdiff-saddle}\nabla_x\phi(x^k,y^k)\in\partial_{\varepsilon_k}f(x^k) \quad\forall k\in\mathbb{N}.
\end{equation}
Moreover, if \ref{ass:nondegen} and~\ref{ass:SMFCQ} hold, then for any $p>1$ there exists $k_0\in\mathbb{N}$ such that 
\begin{equation}\label{eq:Ieps-saddle}
I_{\text{eps}}(x^k, y^k) := \left\{i\in I\colon f(x^k)-f_i(x^k)\leq\varepsilon_k^{(p-1)/p}\right\}=I(x^*)\text{~for all~} k\geq k_0.
\end{equation}
\end{corollary}
\begin{proof}
Let $z\in\Hilbert$ be arbitrary. Using convexity of $f_i$ and $(y^k)\subseteq\Delta^N$, it follows that
   \begin{align*}
       \langle\nabla_x\phi(x^k,y^k),z-x^k\rangle -\varepsilon_k &= \sum_{i\in I} y^k_i\langle\nabla f_i(x^k),z-x^k\rangle - \varepsilon_k\\
       &\leq \sum_{i\in I} y^k_i(f_i(z) - f_i(x^k)) - \varepsilon_k\\
       &\leq f(z) - f(x^k) + f(x^k) - \phi(x^k,y^k) - \varepsilon_k \\
       &= f(z) - f(x^k),
   \end{align*}
which establishes \eqref{eq:e-subdiff-saddle}. By  Theorem~\ref{thm:reform}\ref{item:opt}, we have $\nabla_x\phi(x^*,y^*)=0$ and $f(x^*)-\phi(x^*,y^*)=0$. Hence, due to continuity of $f$, $\phi$ and $\nabla_x\phi$, it follows that $\varepsilon_k=f(x^k)-\phi(x^k,y^k)\to 0$  and$$\dist(0,\partial_{\varepsilon_k}f(x^k))\leq\|\nabla_x\phi(x^k,y^k)\|\to 0\text{~as~}k\to\infty. $$
Moreover, when  \ref{ass:nondegen} and~\ref{ass:SMFCQ} hold, \eqref{eq:Ieps-saddle} then follows by applying Theorem~\ref{thm:eps-subdiff}.
\end{proof}

\subsection{Identification Functions}

In this section, we discuss another approach to support identification based on the results from~\cite{facchinei1998accurate, oberlin2006active}. This will provide an alternative to identifying the sets $I(x^*)$ and $I^+(x^*)$, which appear in Section~\ref{sec:exp} alongside the set from Corollary~\ref{cor:eps-subdiff}. Since $I(x^k)$ cannot be used to identify $I(x^*)$ unless $\dist(0,\partial f(x^k))\to 0$, the general approach taken in this section is to instead take some tolerance $\sigma>0$ and consider the set $\{i\in I\colon f(x^k) - f_i(x^k)\leq\sigma\}$. From the definition, if we take a sufficiently small $\sigma$ (eg, $\sigma<\min_{j\notin I(x^*)}\{f(x^*)-f_j(x^*)\}$), then for $k$ sufficiently large we would have $f(x^k)-f_i(x^k)<\sigma$ for $i\in I(x^*)$ and $f(x^k)-f_j(x^k)>\sigma$ for $j\notin I(x^*)$ as a result of continuity. However it is unclear how to determine such a tolerance \textit{a priori}.

\begin{theorem}\label{thm:oplus}
Consider $f=\max\{f_1,\dots,f_N\}$ and $\phi$ as in~\eqref{eq:minmax} and~\eqref{eq:saddle}, respectively. Let 
\begin{equation}\label{eq:Iplus}
I_{\oplus}(x,y) := \{i\in I\colon f(x)-f_i(x)\leq y_i\}.
\end{equation}There exists a neighbourhood $\mathcal{N}$ of a saddle point $(x^*,y^*)$ of $\phi$ on $\Hilbert\times\Delta^N$ such that the following assertions hold.
    \begin{enumerate}[label=(\alph*)]
        \item\label{item:a} $I_{\oplus}(x,y)\subseteq I(x^*)$ for all $(x,y)\in\mathcal{N}$.
        \item\label{item:b} If~\ref{ass:nondegen} holds and $(x^*,y^*)$ satisfying the strict complementarity condition: 
        \begin{equation}\label{eq:strict complementarity}
        y^*_i>0\iff f(x^*)=f_i(x^*)\quad\forall i\in I,
        \end{equation}
        then $I_{\oplus}(x,y) = I(x^*) = I^+(x^*)$ for all $(x,y)\in\mathcal{N}$.
        \item\label{item:c} If \ref{ass:SMFCQ} holds, then $I^+(x^*)\subseteq I_{\oplus}(x,y)\subseteq I(x^*)$ for all $(x,y)\in\mathcal{N}$.
    \end{enumerate}
\end{theorem}
\begin{proof}
    \ref{item:a} For $j\notin I(x^*)$, we have $f(x^*) - f_j(x^*)>0=y^*_j$. By continuity of $f$ and $f_j$, there exists a neighbourhood $\mathcal{N}$ of $(x^*,y^*)$ such that $f(x)-f_j(x)>y_j$ for all $(x,y)\in\mathcal{N}$. The latter implies $j\not\in I_{\oplus}(x,y)$ for all $(x,y)\in\mathcal{N}$, from which the result follows.
    \ref{item:b} For $i\in I^+(x^*)$, we have $f(x^*) - f_i(x^*) = 0 < y^*_i$. By continuity of $f$ and $f_i$, there exists a neighbourhood $\mathcal{N}$ of $(x^*,y^*)$ such that $f(x)-f_i(x)\leq y_i$ for all $(x,y)\in\mathcal{N}$ and hence $I^+(x^*)\subseteq I_{\oplus}(x,y)$. By Theorem~\ref{thm:nondegeneracy}, the strict complementarity condition~\eqref{eq:strict complementarity} is equivalent to $I(x^*)=I^+(x^*)$. Combining these two inclusions with part~\ref{item:a} establishes the result.
    \ref{item:c} Since $\Lambda(x^*)=\{y^*\}$ under \ref{ass:SMFCQ}, we have
    $$I^+(x^*) = \{i\in I\colon\exists~y^*\in\Lambda(x^*)\text{~s.t.~}y^*_i>0\}=\{i\in I\colon y^*_i>0\}.$$
    Let $i\in I^+(x^*)$, then $y^*_i>0=f(x^*) -f_i(x^*)$. Similarly to part~\ref{item:b}, we therefore have $I^+(x^*,y^*)\subseteq I_{\oplus}(x,y)$ for some neighbourhood $\mathcal{N}$ of $(x^*,y^*)$ and all $(x,y)\in\mathcal{N}$. Once again combining with part~\ref{item:a} completes the proof.
\end{proof}


Let $S$ denote the set of saddle points of $\phi$ on $\Hilbert\times\Delta^N$. A continuous function $\rho\colon\Hilbert\times\mathbb{R}^N\to\mathbb{R}_+$ is said to be an \textit{identification function}~\cite{facchinei1998accurate} for $S$ if $\rho(x^*,y^*)=0$ for all $(x^*,y^*)\in S$, and 
\begin{equation}\label{eq:limit}
\lim_{\substack{(x,y)\to(x^*,y^*)\in S\\ (x,y)\notin S}}~\frac{\rho(x,y)}{\dist((x,y), S)} = +\infty. 
\end{equation}
Informally, an identification function is a function $\rho\colon\Hilbert\times\Delta^N\to\mathbb{R}_+$ for which $\rho(x,y)$ converges to $0$ more slowly than $(x,y)$ converges to $(x^*,y^*)$. Generally, these functions can be constructed by taking a function with a (local) Lipschitz property, and raising it to some power between $0$ and $1$. This is illustrated in the following example.

\begin{example}[Identification functions]\label{ex:id funcs}
Let $\|\cdot\|$ be any norm and $0<\gamma<1$. Then the following are functions are  identification functions for $S$.
    \begin{enumerate}[(a)]
    \item $\rho_1(x,y) = \left(\left\|\sum_{i=1}^N y_i \nabla f_i(x)\right\| + f(x) - \phi(x,y)\right)^\gamma.$
    \item $\rho_2(x,y; \lambda) = \|P_{K} (x - \lambda F(x))\|^\gamma$, with $K$ and $F$ as defined as in~\eqref{eq:VIP}.
        \item $\rho_3(x,y) = \left\|\begin{pmatrix}
    \sum_{i=1}^N y_i\nabla f_i(x)\\
    y_1( f(x) - f_1(x) )\\
    \vdots\\
    y_N( f(x) - f_N(x) )
\end{pmatrix}\right\|^\gamma.$
    \end{enumerate}
\end{example}

In the previous section, we used the $\varepsilon$-subdifferential to define a dynamic tolerance that can be applied to support identification. The estimate used there (with $\varepsilon_k=f(x^k)-\phi(x^k,y^k)$) is not necessarily an identification function for two reasons: first, this may not satisfy the limit condition in~\eqref{eq:limit}; second, $\varepsilon_k=0$ on its own does not imply $(x^k,y^k)=(x^*,y^*)$. This is why the functions listed in Example~\ref{ex:id funcs} incorporate both terms from Theorem~\ref{thm:reform}\ref{item:opt conditions}.

Note that the following result does not require nondegeneracy.

\begin{theorem}
Consider $f=\max\{f_1,\dots,f_N\}$ and $\phi$ as in~\eqref{eq:minmax} and~\eqref{eq:saddle}, respectively, and suppose $x^*\in\argmin f$ is unique. Let $\rho$ be an identification function for $S$. Then, for all $y^*\in\Lambda(x^*)$, there exists a neighbourhood $\mathcal{N}$ of $(x^*,y^*)$ such that the following assertions hold.
    \begin{enumerate}[label=(\alph*)]
        \item\label{item:func ident} For all $(x,y)\in\mathcal{N}$, we have $$\mathcal{A}(x,y) := \{i\in \colon f(x) - f_i(x)\leq\rho(x,y)\} = I(x^*).$$
        \item\label{item:strong func ident} If~\ref{ass:SMFCQ} holds then, then, for all $(x,y)\in\mathcal{N}$, we have
        $$\mathcal{A}_+(x,y) := \{i\in I\colon f(x)-f_i(x)\leq\rho(x,y)\leq y_i\} = I^+(x^*).$$
    \end{enumerate}
\end{theorem}
    \begin{proof}
     We include the prove of both parts for completeness, although the arguments similar to  \cite[Theorems 2.2/2.4]{facchinei1998accurate} respectively. \ref{item:func ident} First, we note that, due to convexity, $f_i$ is locally Lipschitz for all $i\in I$ (see \cite[Theorem 4.1.3]{Borwein2006}). Thus there exists a neighbourhood $\mathcal{N}_x$ of $x^*$ on which $f_i$ is $L$-Lipschitz. It follows that $f$ is $L$-Lipchitz (see \cite[Proposition 9.10]{rockafellar2009variational}), and $f-f_i$ is $2L$-Lipschitz (see \cite[Exercise 9.8(a)\&(b)]{rockafellar2009variational}) on the same neighbourhood.
     Therefore, for all $i\in I$ and $x\in\mathcal{N}_x$, we have
     \begin{equation}\label{eq:lip est}
         f(x) - f_i(x)\leq f(x^*)-f_i(x^*) + 2L\|x-x^*\|.
     \end{equation}
     Now let $i\in I(x^*)$. Since $\argmin f=\{x^*\}$, it follows that
     \begin{equation}\label{eq:lip est2}
         f(x) - f_i(x)\leq 2L\|x-x^*\|\leq 2L\sqrt{\|x-x^*\|^2 + \dist(y,\Lambda(x^*))^2}= 2L\dist((x,y), S), 
     \end{equation}
     for all $(x,y)\in\mathcal{N}_x\times\Delta^N$. Now, as a result of \eqref{eq:limit}, there exists a neighbourhood $\mathcal{N}\subseteq\mathcal{N}_x\times\Delta^N$ of $(x^*,y^*)$ such that $\rho(x,y)\geq 2L\dist((x,y), S)$ for all $(x,y)\in\mathcal{N}$. Combining this with \eqref{eq:lip est2} shows that $i\in\mathcal{A}(x,y)$, and hence we conclude that $I(x^*)\subseteq\mathcal{A}(x,y)$.
     Conversely, for $j\notin I(x^*)$, $f(x^*)-f_j(x^*)>0$. Then, by continuity of $f$ and $f_i$, there exists a neighbourhood $\mathcal{N}$ of $(x^*,y^*)$ such that $f(x)-f_j(x)>0$ for all $(x,y)\in\mathcal{N}$. This shows that $j\not\in\mathcal{A}(x,y)$ and hence $\mathcal{A}(x,y)\subseteq I(x^*)$.
     \ref{item:strong func ident} Let $i\in I^+(x^*)$. Since $\Lambda(x^*)=\{y^*\}$ under \ref{ass:SMFCQ}, we have
    $$I^+(x^*) = \{i\in I\colon\exists~y^*\in\Lambda(x^*)\text{~s.t.~}y^*_i>0\}=\{i\in I\colon y^*_i>0\}.$$ Since $\rho(x,y)\to 0$ as $(x,y)\to(x^*,y^*)$, there exists a neighbourhood $\mathcal{N}$ of $(x^*,y^*)$ such that $\rho(x,y)\leq y_i$ for all $(x,y)\in\mathcal{N}$. Using~\ref{item:func ident} and shrinking the neighbourhood $\mathcal{N}$ if necessary, we deduce that $f(x)-f_i(x)\leq\rho(x,y)$. Putting this altogether gives
     $$ f(x)-f_i(x)\leq \rho(x,y)\leq y_i \quad\forall (x,y)\in\mathcal{N}.$$
     Thus $i\in\mathcal{A}_+(x,y)$, and hence we conclude that $I^+(x^*)\subseteq\mathcal{A}_+(x,y)$. 
     Conversely, let $j\notin I^+(x^*)$. As a result of~\eqref{eq:limit}, there exists a neighbourhood $\mathcal{N}$ of $(x^*,y^*)$ such that
     $$\rho(x,y)\geq 2\|(x,y)-(x^*,y^*)\|\quad\forall (x,y)\in\mathcal{N}.$$
     Then, since $y^*_j=0$ due to \ref{ass:SMFCQ}, we have
     $$y_j=\lvert y_j-y^*_j\rvert\leq \|(x,y)-(x^*,y^*)\|<\rho(x,y) \quad\forall (x,y)\in\mathcal{N}\setminus\{(x^*,y^*)\}.$$
     Thus $j\not\in\mathcal{A}_+(x,y)$ when $(x,y)\in\mathcal{N}\setminus\{(x^*,y^*)\}$.
     Also, $\mathcal{A}_+(x^*,y^*) = I^+(x^*)$ since $\rho(x^*,y^*)=0$, and so $j\not\in \mathcal{A}_+(x^*,y^*)$. 
     Altogether, $j\notin\mathcal{A}_+(x,y)$ for all $(x,y)\in\mathcal{N}$. This shows that $\mathcal{A}_+(x,y)\subseteq I^+(x^*)$, which completes the proof.
    \end{proof}

\section{Experiments}\label{sec:exp}

In this section, we presentation numerical experiments which compare the different support measures developed in Section~\ref{sec:identification}. Given an algorithm for solving~\eqref{eq:VIP}, our approach is loosely described as follows. First, we run some number $k_0$ of iterations of the algorithm to solve~\eqref{eq:saddle}, after which we assume that $(x^{k_0},y^{k_0})$ is close enough to $(x^*,y^*)$ so that $I(x^*)$ can be reasonably approximated. At this point, we ``correct'' the support $I$ by taking a estimation $\Tilde{I}\approx I(x^*)$. We then restart the same algorithm, from $x^{k_0}$, and solve the reduced form of~\eqref{eq:saddle} where $I$ is replaced by $\Tilde{I}$. Note that due to the dimension reduction, the variable $y^{k_0}$ may no longer be feasible and must hence be reset. In practice, we would expect the algorithm to converge faster when applied to the simpler reduced problem, and this is what we will observe in this section.

Although results of Section~\ref{sec:identification} show that we can precisely identify $I(x^*)$ after some number $k_0$ of iterations, in practice we have no way of knowing whether $k_0$ is large enough. This may lead to false positives ($\Tilde{I}\setminus I(x^*)$) or false negatives ($I(x^*)\setminus\Tilde{I}$). If there are many false positives, then the purpose of the support correction is defeated since the new problem is not much simpler when compared to the original. On the other hand, even one false negative can mean the solution of the reduced model is different to that of the original. This is why we should perform this support correction process more than once, so that if false positives/negatives appear in one approximation, they are hopefully removed in the next. Two versions of this heuristic, deterministic and stochastic, are stated respectively in Algorithms~\ref{alg:det} and~\ref{alg:stoc}. There, we denote by $\mathcal{F}$ a single step of a numeric solver for~\eqref{eq:saddle}, that is, the recursive sequence $(x^{k+1},y^{k+1}) = \mathcal{F}(x^k,y^k)$ converges to a saddle point $(x^*,y^*)$.

\begin{algorithm2e}
    \SetKwInOut{Input}{Input}
    \Input{Solver $\mathcal{F}$, support measure $\overline{I}\colon\Hilbert\times\Delta^N\to 2^I$, sequence $(k_j)_{j\in\mathbb{N}}\subset\mathbb{N}$ of iteration counts}
    \caption{Deterministic--Support Correction Heuristic (D-SCH)}\label{alg:det}
    \textbf{Initialisation} Choose $x^0\in\Hilbert$, $y^0\in\Delta^N$ \;
    Define initial objective $\phi(x,y) = \sum_{i\in I}y_i f_i(x)$ and constraints $\Delta=\Delta^N$ \;
    \For{$j=0,1,\dots$}{
    \For{$k=1,\dots,k_j$}{$(x^k,y^k) = \mathcal{F}(x^{k-1},y^{k-1})$\;}
    Take support measurement $\Tilde{I}=\overline{I}(x^{k_j}, y^{k_j})$ \;
    Define new objective $\phi(x,y) = \sum_{i\in\Tilde{I}}y_i f_i(x)$ and constraints $\Delta = \Delta^{\lvert\Tilde{I}\rvert}$ \;
    Set $x^0:=x^{k_j}$, choose $y^0\in\Delta$ \;
    }
\end{algorithm2e}

\begin{algorithm2e}
    \SetKwInOut{Input}{Input}
    \Input{Solver $\mathcal{F}$, support measure $\overline{I}\colon\Hilbert\times\Delta^N\to 2^I$, $\delta\in(0,1)$}
    \textbf{Initialisation} Choose $x^0\in\Hilbert$, $y^0\in\Delta^N$, set $q=1$\;
    \caption{Stochastic--Support Correction Heuristic (S-SCH)}\label{alg:stoc}
    \For{$k=1,2,\dots$}{
    $(x^k,y^k) = \mathcal{F}(x^{k-1},y^{k-1})$\;
    \Prob{$1-q$}{
        Set $\Tilde{I}=\overline{I}(x^k, y^k)$ \;
        Take support measurement $\Tilde{I}=\overline{I}(x^{k_j}, y^{k_j})$ \;
        Define new objective $\phi(x,y) = \sum_{i\in\Tilde{I}}y_i f_i(x)$ and constraints $\Delta = \Delta^{\lvert\Tilde{I}\rvert}$ \;
        Set $x^0:=x^k$, choose $y^0\in\Delta$ \;
        Reset $q=1$ \;
    }{
    $q=\delta q$ \;
    }
    }
\end{algorithm2e}

In the following experiments, we apply both Algorithm~\ref{alg:det} and Algorithm~\ref{alg:stoc} with the following support measures
\begin{align*}
    I(x) &= \left\{i\in I\colon f_i(x) = f(x)\right\}\\
    I_{\oplus}(x,y) &= \left\{i\in I\colon f(x) - f_i(x) \leq y_i\right\}\\
    I_{\text{eps}}(x^k, y^k) &= \left\{i\in I\colon f(x^k)-f_i(x^k)\leq\sqrt{\varepsilon_k}\right\},
\end{align*}
where $\varepsilon_k=f(x^k)-\phi(x^k,y^k)$. Observe that $I$ above is the ``naive'' approach and not guaranteed to correctly estimate $I(x^*)$, unless $\dist(0,\partial f(x^k))\to 0$ as shown in Theorem~\ref{thm:subgradient}. Meanwhile, $I_{\oplus}$ and $I_{\text{eps}}$ are theoretically justified as shown in Theorem~\ref{thm:oplus} and Corollary~\ref{cor:eps-subdiff} respectively (with $p=2$ in Corollary~\ref{cor:eps-subdiff}). While these estimations have been shown to identify the support whenever the number of iterations is sufficiently large, our heuristic is not particularly useful if the sequence $(x^k,y^k)$ is already very close to the solution $(x^*,y^*)$. This is why we perform the support measurement step with a small positive tolerance $\sigma>0$. For instance, we can pad the right hand sides of the inequalities in $I_{\oplus}$ and $I_{\text{eps}}$, and replace the naive measurement by $I(x) = \{i\in I\colon f_i(x) \geq f(x) - \sigma\}$. Depending on the tolerance chosen, these measurements may include a small proportion of additional false positives, but they have the desired effect of being able to reduce the problem before convergence.

The base method for our heuristic is the adaptive Golden RAtio ALgorithm (aGRAAL) \cite{malitsky2020golden}: a method for solving~\eqref{eq:VIP} for locally Lipschitz $F$ without requiring a backtracking linesearch. We compare this against the subgradient method~\eqref{eq:subgrad}, with step-sizes defined by $\frac{\gamma_k}{\|g^k\|}$ for decaying $\gamma_k\in\mathbb{R}_{++}$, noting that only $I(x)$ is well defined here since the sequence $(y^k)\subset\Delta^N$ is not generated. In each case, we compare our heuristic against the base method, that is, without the heuristic attached.

Our experiments are run in Python 3.9.0, on a Machine with an Intel(R) Core(TM) i7-10510U CPU @ 1.80GHz processor and 8GB memory. We compare the base methods, subgradient and aGRAAL, against Algorithms~\ref{alg:det} and~\ref{alg:stoc}. After performing the support correction step at a specified number of $k_0$ iterations in Algorithm~\ref{alg:det}, we run the algorithms further to examine the accuracy after each distinct support correction.

\subsection{Piecewise Linear}
Our first example is a piecewise linear problem.
\begin{equation}\label{eq:linear}
    \min_{x\in\mathbb{R}^n}\max_{i\in I}\langle\alpha_i,x\rangle+\beta_i.
\end{equation}

This is a simple problem which we are using as a proof of concept, that is, to confirm that our heuristic performs as expected, although problems of this form can be applied in practice (see, for instance, \cite[UOM1/UOM2]{ZAMIR2015947}). This is also a useful problem to study since we can generate the actual solution $x^*\in\argmin f$ via linear programming. This will be applied to determine the accuracy of our support measurements, and the accuracy of the generated sequences in terms of decay of the objective function $f(x^k) - f(x^*)$. In the above equation~\eqref{eq:linear} $\alpha_i\in\mathbb{R}^n,\beta_i\in\mathbb{R}$ are randomly generated from a standard multivariate normal distribution. 

\begin{figure}[h!]
    \centering
    \includegraphics[width=.8\textwidth]{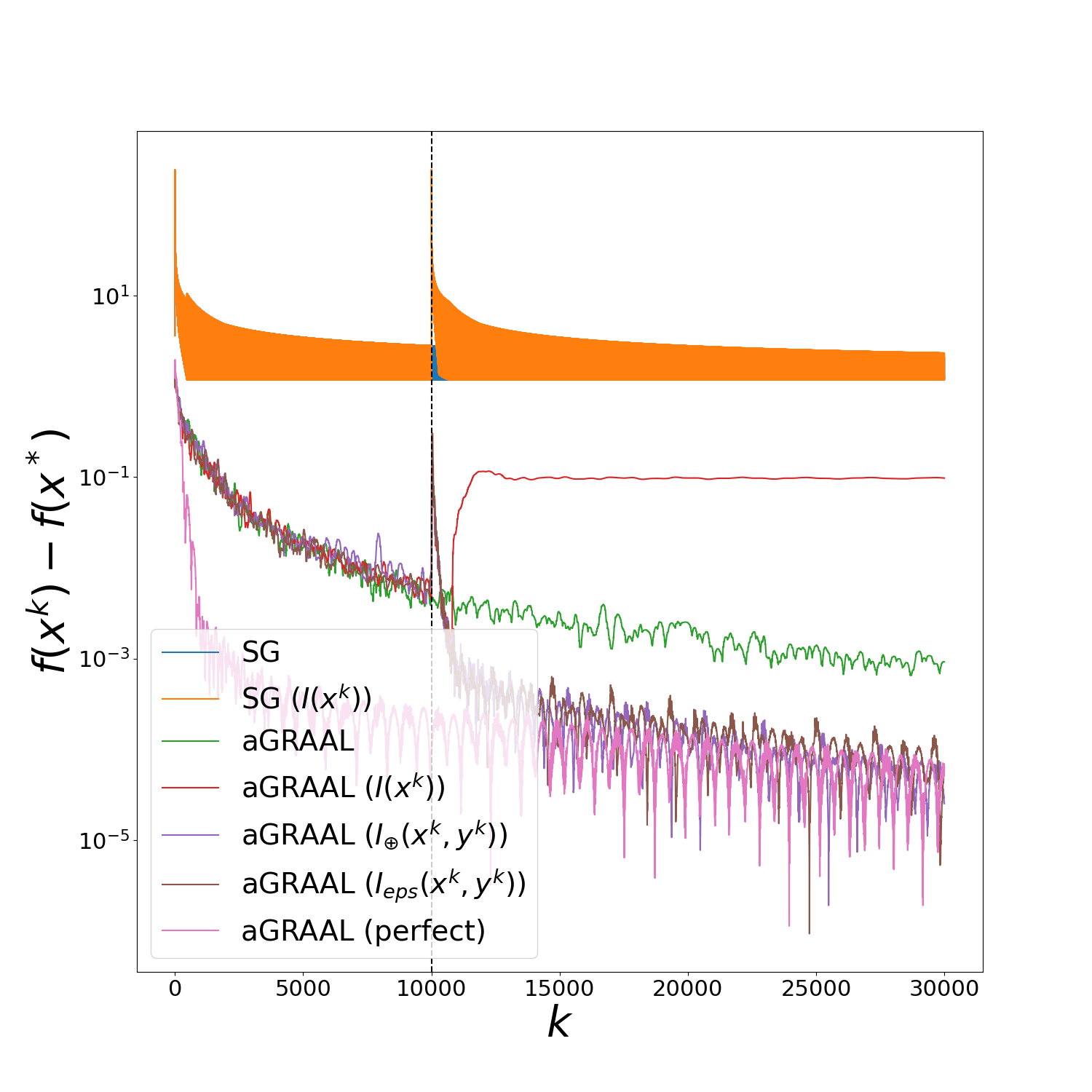}
    \caption{Objective function decay over iterations, for Algorithm~\ref{alg:det} applied to the piecewise linear problem~\eqref{eq:linear}.}
    \label{fig:linear-obj}
\end{figure}
\begin{figure}[h!]
    \centering
    \includegraphics[width=\textwidth]{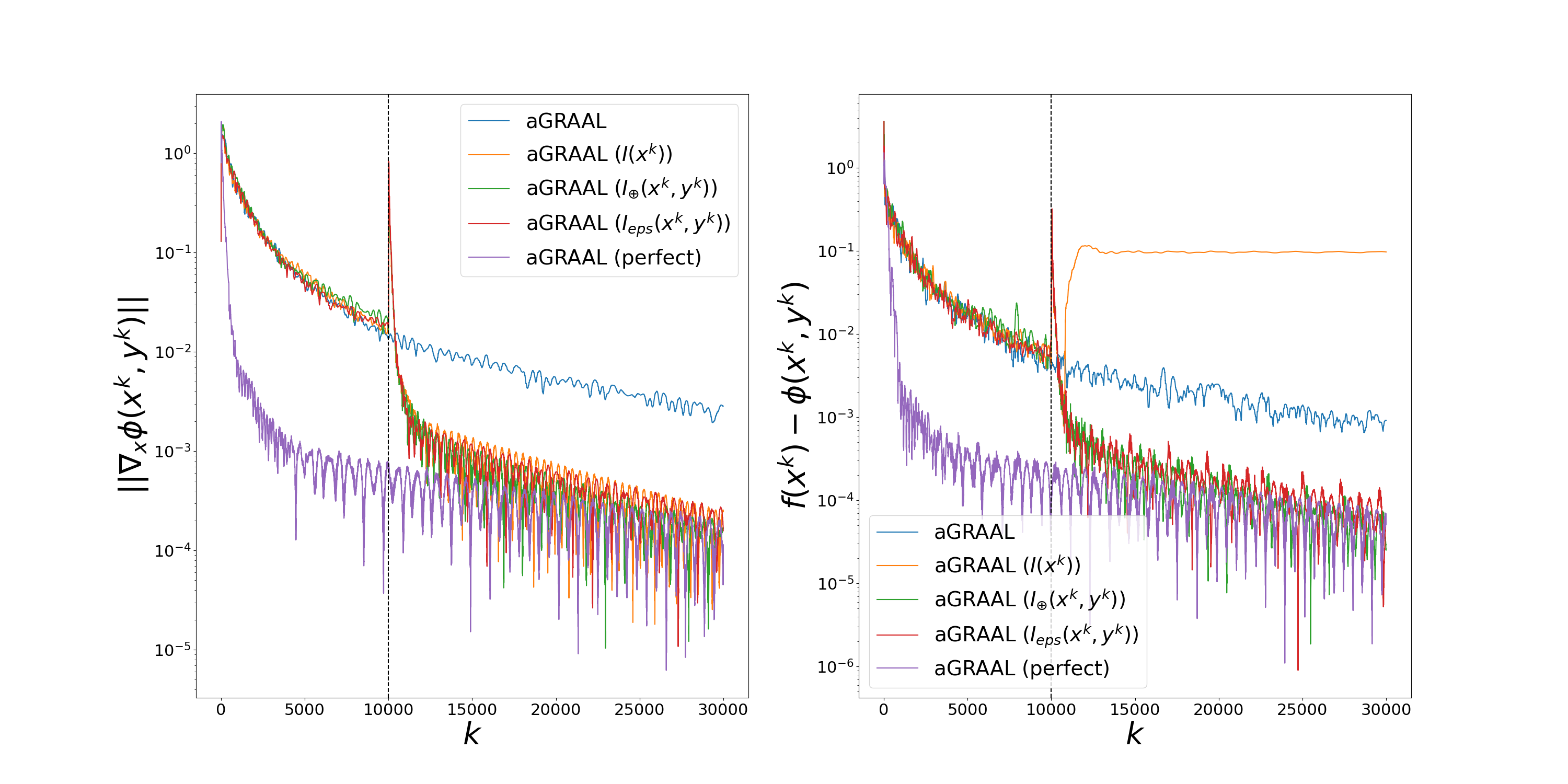}
    \caption{Error terms as derived in Theorem~\ref{thm:reform} for Algorithm~\ref{alg:det} over iterations, applied to the piecewise linear problem~\eqref{eq:linear}.}
    \label{fig:linear-err}
\end{figure}

Our results are displayed in Figures~\ref{fig:linear-obj} and~\ref{fig:linear-err}, with $N=2200,n=45$. We ran both the subgradient methods and aGRAAL for $k_0=10000$ initial iterations, and reduced by taking one of several measurements as described above. Figure~\ref{fig:linear-obj} shows the decay in objective function for both aGRAAL and~\eqref{eq:subgrad}, while Figure~\ref{fig:linear-err} shows the accuracy measurements derived for $(x^k,y^k)$ in Theorem~\ref{thm:reform}, which measure the distance to a saddle point $(x^*,y^*)$.

We first note that the subgradient method becomes stuck very early, in terms of the accuracy in the objective function. This is likely due to the sequence $\gamma_k$ decaying too quickly in the step-size $\lambda_k=\frac{\gamma_k}{\|g^k\|}$. As a result, the measurement $I(x^{k_0})$ is inaccurate. The plots also show a significant improvement in accuracy for aGRAAL after the support is measured. From both Figures~\ref{fig:linear-obj} and~\ref{fig:linear-err}, we can observe that after measuring $I(x^{k_0})$ the sequence converges to the wrong point, and indeed, we report that $I(x^{k_0})$ contained 21 false negatives, even with tolerance of $\sigma=10^{-2}$. Meanwhile, $I_{\oplus}(x^{k_0},y^{k_0})$, also with $\sigma=10^{-2}$, returned a measurement with only 3 false negatives and one false positive. As we can see from both figures, this was sufficiently close to the correct measurement thus resulting in the observed speedup and convergence. $I_{\text{eps}}$ performed even more desirably, which gave 10 false positives but no false negatives, and with a tolerance $\sigma=0$. 

Since in this case we have knowledge of the true index set $I(x^*)$,  we also applied aGRAAL to the reduced problem initially to examine the cost of perfect information. The variant of aGRAAL with perfect information immediately speeds up compared to the raw version, converging to a level of $10^{-3}$ in under 5,000 iterations. At the support correction step at $k=10,000$ iterations, there is a brief spike in accuracy, likely due to the required resetting of $y^k$ to become feasible in the reduced problem. Shortly after this step, the variants with supports measured by $I_{\oplus}$ and $I_{\text{eps}}$ converge to the same accuracy as that of the variant with perfect knowledge. However, the measurement $I(x^k)$ contains many false negatives even for $x^k$ generated by aGRAAL, and as a result the sequence is converging to the wrong point in this version of the heuristic.

Our next experiment shows the accuracy of each support measurement over many random problem instances of different sizes. For each instance, we can determine the actual support $I(x^*)$ via linear programming. The following table shows the accuracy of each support measure in general, in terms of false positives/false negatives. For each instance, we compared all measures at a low (5,000) and high (30,000) number of iterations of aGRAAL. Our comparisons include $I,I_{\oplus},I_{\text{eps}}$, as well as
\begin{align*}
    \mathcal{A}(x,y) &:= \{i\in I\colon f(x) - f_i(x)\leq\rho(x,y)\}\\
    \mathcal{A}_+(x,y) &:=\{i\in I\colon f(x) - f_i(x)\leq\rho(x,y)\leq y_i\},
\end{align*}

with the identification functions
\begin{align*}
\rho_1(x,y) &= \left(\left\|\sum_{i\in I} y_i\nabla f_i(x)\right\|_1 + f(x) - \phi(x,y)\right)^\gamma\\
\rho_2(x,y; \lambda) &:= \|P_{\Hilbert\times\Delta^N}((x,y) - \lambda F(x,y))\|^\gamma
\end{align*}
for $\gamma=0.8$. Note that $\rho_1$ and $\rho_3$ in Example~\ref{ex:id funcs} are in fact equal if the chosen norm is $\|\cdot\|_1$.
\begin{table}[h!]
    \centering
    \begin{tabular}{c|c|c|c|c|c|c|c|c|c}
    $N$ & $n$ & $k$ & $I$ & $I_{\oplus}$ & $I_{\text{eps}}$ & $\mathcal{A}~(\rho_1)$ & $\mathcal{A}_+~(\rho_1)$ & $\mathcal{A}~(\rho_2)$ & $\mathcal{A}_+~(\rho_2)$ \\\midrule
    500 & 5 & 5000    & 0/5  & 0/3  & 0/0  & 0/0   & 0/0   & 0/0  & 0/4  \\
    500 & 5 & 30000   & 0/0  & 0/0  & 0/0  & 0/0   & 0/0   & 0/0  & 0/1  \\
    1000 & 5 & 5000   & 0/5  & 0/4  & 3/0  & 1/0   & 1/0   & 0/0  & 0/5  \\
    1000 & 5 & 30000  & 0/3  & 0/0  & 0/0  & 0/0   & 0/0   & 0/0  & 0/2  \\
    1500 & 5 & 5000   & 0/5  & 0/0  & 8/0  & 9/0   & 9/0   & 4/0  & 0/5  \\
    1500 & 5 & 30000  & 0/2  & 0/0  & 1/0  & 1/0   & 1/0   & 1/0  & 0/1  \\
    2000 & 5 & 5000   & 0/5  & 0/1  & 5/0  & 9/0   & 9/0   & 6/0  & 0/4  \\
    2000 & 5 & 30000  & 0/1  & 0/0  & 3/0  & 2/0   & 2/0   & 1/0  & 0/1  \\
    2500 & 10 & 5000  & 1/9  & 1/3  & 7/0  & 31/0  & 31/0  & 12/0 & 0/10 \\
    2500 & 10 & 30000 & 0/0  & 0/0  & 3/0  & 2/0   & 2/0   & 1/0  & 0/3  \\
    3000 & 10 & 5000  & 0/9  & 1/4  & 6/0  & 11/0  & 11/0  & 4/0  & 0/8  \\
    3000 & 10 & 30000 & 0/4  & 0/1  & 3/0  & 0/0   & 0/0   & 0/0  & 0/6  \\
    3500 & 20 & 5000  & 0/18 & 0/9  & 4/0  & 52/0  & 52/0  & 5/0  & 0/21 \\
    3500 & 20 & 30000 & 0/5  & 0/0  & 2/0  & 5/0   & 5/0   & 2/0  & 0/16 \\
    4000 & 20 & 5000  & 0/20 & 0/15 & 9/0  & 69/0  & 69/0  & 9/0  & 0/21 \\
    4000 & 20 & 30000 & 1/8  & 1/0  & 1/0  & 3/0   & 3/0   & 1/0  & 0/14 \\
    4500 & 50 & 5000  & 1/50 & 1/32 & 13/1 & 126/0 & 126/0 & 13/1 & 0/51 \\
    4500 & 50 & 30000 & 2/0  & 2/0  & 6/0  & 11/0  & 11/0  & 4/0  & 0/34 \\
    5000 & 50 & 5000  & 0/46 & 1/30 & 18/0 & 175/0 & 175/0 & 29/0 & 0/51 \\
    5000 & 50 & 30000 & 1/1  & 1/1  & 5/0  & 14/0  & 14/0  & 2/0  & 0/40
    \end{tabular}
    \caption{Accuracy of each support measurement in terms of false positives/negatives.}
    \label{tab:test}
\end{table}
From Table~\ref{tab:test} we can see that $I_{\text{eps}}$ is clearly the most accurate out of all measurements. Both $I$ and $I_{\oplus}$ have a very low false positive count, but the high false negative count is undesirable since the solutions of the resulting problem would not match those of the original. Meanwhile, the accuracy of the sets $\mathcal{A}$ and $\mathcal{A}_+$ depend heavily on the choice of identification function, but generally have a high false positive or negative count. Out of these, the best performing was $\mathcal{A}$ with $\rho_2$, with a low-medium level of false positives and no false negatives, except for one in the case of $N=4500,n=50,k=5000$. Note that this is the low iterations mark. Meanwhile, the set $I_{\text{eps}}$ has a very low count of false positives, and no false negatives except for one in the aforementioned instance.

\subsection{Piecewise Quadratic}

\begin{figure}[h!]
    \centering
    \includegraphics[width=\textwidth]{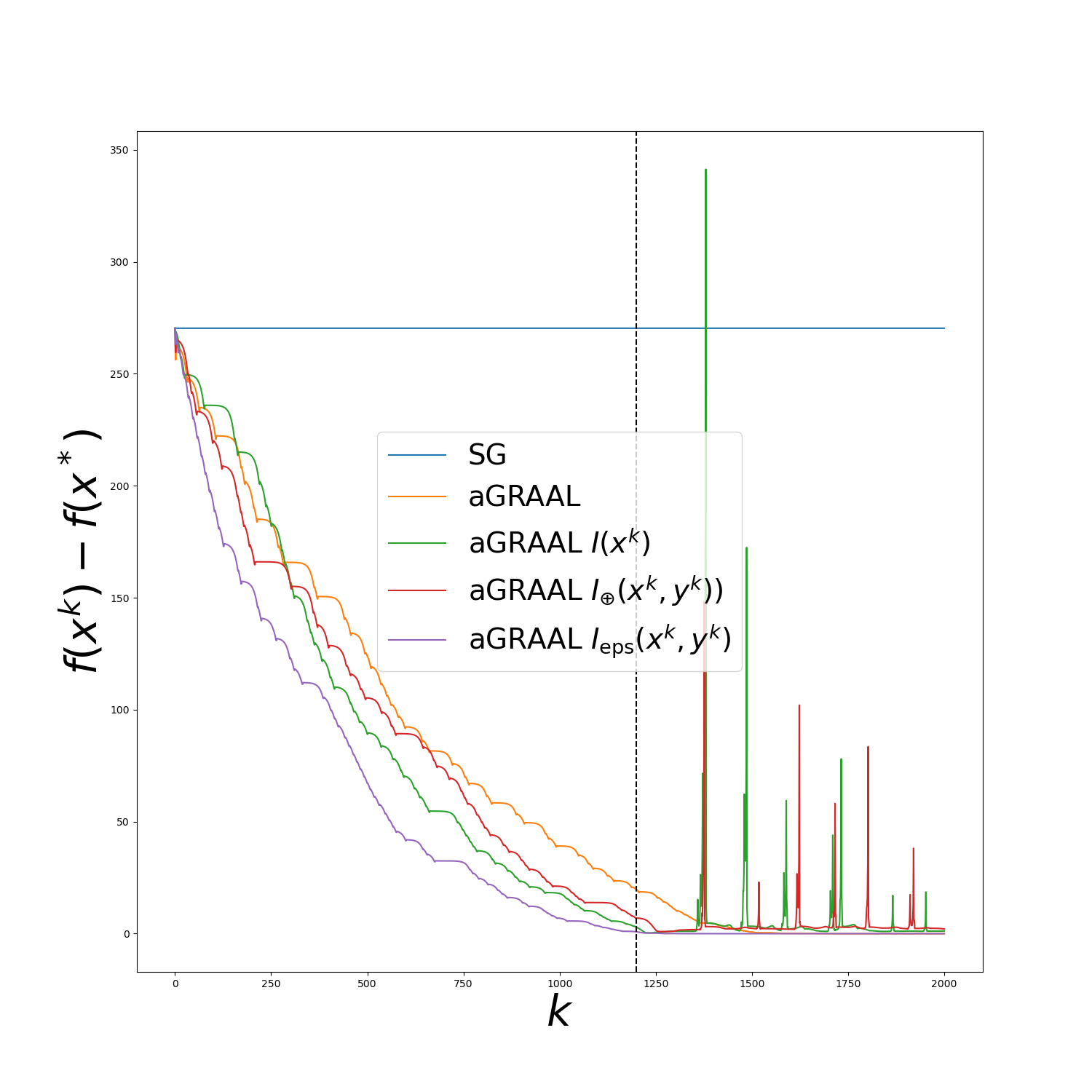}
    \caption{Objective function decay over iterations, for Algorithm~\ref{alg:det} applied to the piecewise quadratic problem~\eqref{eq:quad}.}
    \label{fig:quad-det-obj}
\end{figure}
\begin{figure}[h!]
    \centering
    \includegraphics[width=\textwidth]{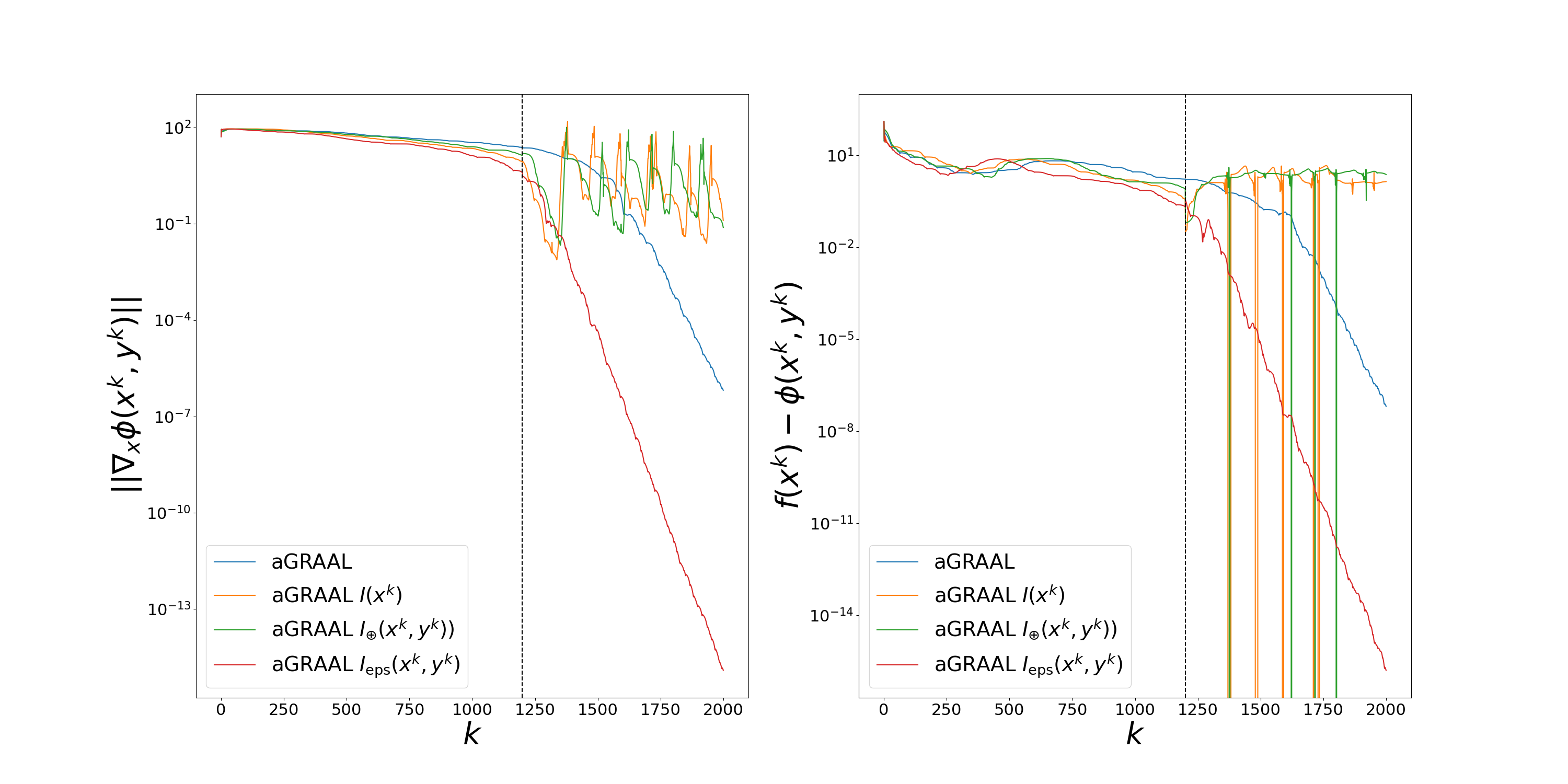}
    \caption{Error terms as derived in Theorem~\ref{thm:reform} over iterations, for Algorithm~\ref{alg:det} applied to the piecewise quadratic problem~\eqref{eq:quad}.}
    \label{fig:quad-det-err}
\end{figure}
\begin{figure}[h!]
    \centering
    \includegraphics[width=\textwidth]{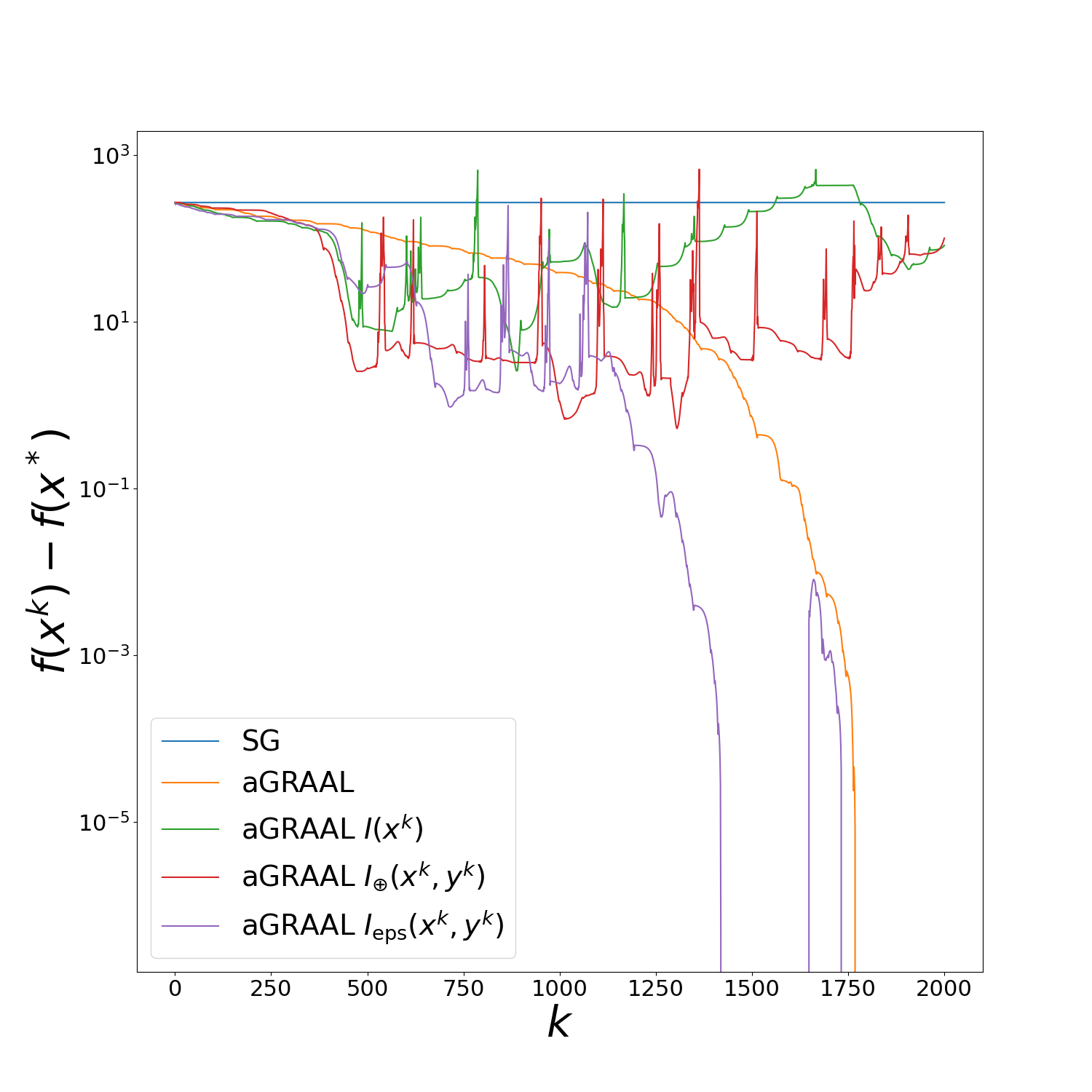}
    \caption{Objective function decay over iterations, for Algorithm~\ref{alg:stoc} applied to the piecewise quadratic problem~\eqref{eq:quad}.}
    \label{fig:quad-stoc-obj}
\end{figure}
\begin{figure}[h!]
    \centering
    \includegraphics[width=\textwidth]{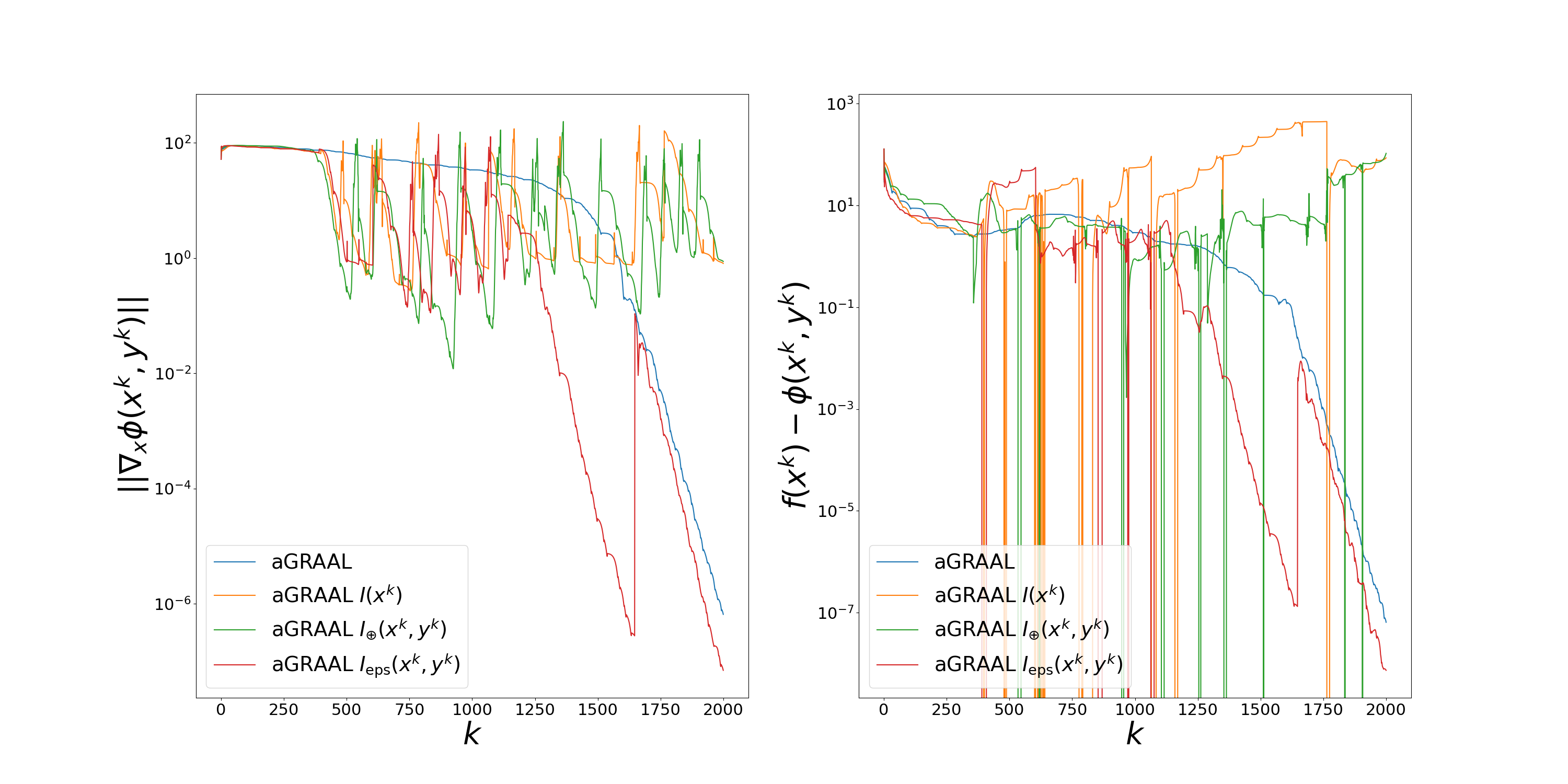}
    \caption{Error terms as derived in~Theorem~\ref{thm:reform} over iterations, for Algorithm~\ref{alg:stoc} applied to the piecewise quadratic problem~\eqref{eq:quad}.}
    \label{fig:quad-stoc-err}
\end{figure}

Our next example is the piecewise quadratic problem
    \begin{equation}\label{eq:quad}
        \min_{x\in\mathbb{R}^n}\max_{i\in I} x^T H_i x + q_i^T x.
    \end{equation}

Similarly to the previous example, this is a generic problem that is also intended as a proof of concept. In this case we have $N=600,n=30$, $H_i\in\mathbb{S}_+^n$ is generated for all $i$ by normally generating $A_i\in\mathbb{R}^{n\times n}$ and setting $H_i= A^T A$, and $q_i$ is generated uniformly in $[-1,1]^n$. In this section we take all support measurements with a tolerance of $\epsilon=10^{-1}$.

We generated our initial point $x^0$ as the mean of the minimisers of $f_i$, which are each computed as the solution to the linear system $2H_i = -q_i$. Unlike in the previous example, only the set $I_{\text{eps}}(x^{k_0},y^{k_0})$ accurately estimated the active support $I(x^*)$. This set contained a large proportion (417, out of the original index set of 600) of subfunctions, and it is likely that many of these are false positives, however there is still some simplification that is achieved as can be seen in Figures~\ref{fig:quad-det-obj} and~\ref{fig:quad-det-err}. Meanwhile, $I(x^{k_0})$ and $I_{\oplus}(x^{k_0},y^{k_0})$ are too small, and did not accurately coincide with $I(x^*)$. So the resulting sequence from these measurements both converge to the wrong point. Similar speedup is observed for the stochastic versions in Figures~\ref{fig:quad-stoc-obj} and~\ref{fig:quad-stoc-err}, where there are slight differences only due to the random nature of the heuristic --- both in terms of the random number of iterations between corrections, and the randomness in the initialisation we apply for the underlying aGRAAL. Here we can also observe the effect of the support correction step: when we measure $\Tilde{I}$ we must reset the sequence $y^k$ since the dimension is reduced, which results in the small ``jumps'' we can see in each Figure, as the sequence pair $(x^k,y^k)$ becomes momentarily displaced. In the stochastic versions, as in the determinant case, only $I_{\text{eps}}$ was able to measure $I(x^*)$ to a sufficent precision, and thus produce a sequence that converges to the correct point.

\subsection{Generalised Spanning Circle}

\begin{figure}[h!]
    \centering
    \includegraphics[width=\textwidth]{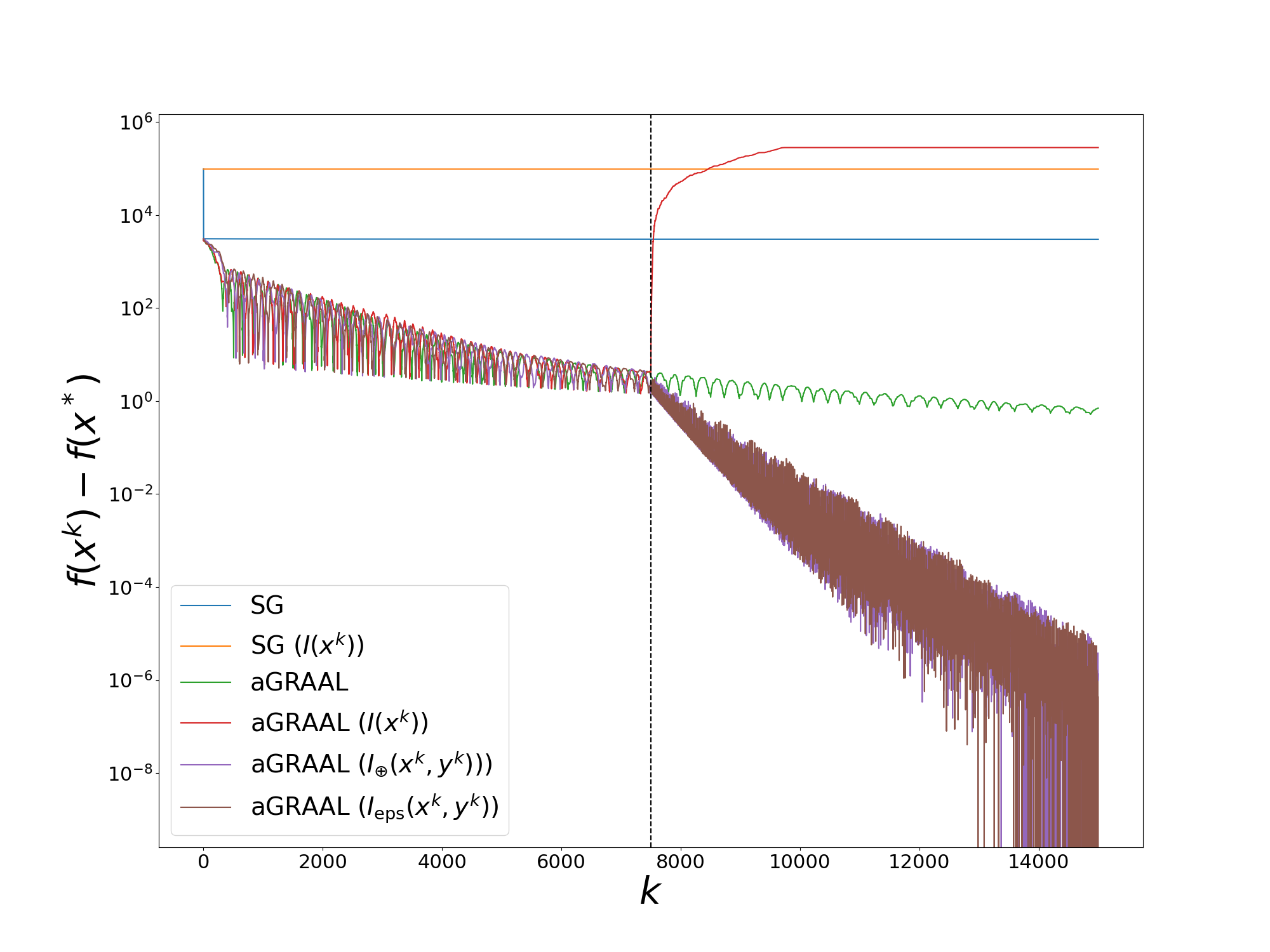}
    \caption{Objective function decay over iterations, for Algorithm~\ref{alg:det} applied to the generalised spanning circle problem~\eqref{eq:weighted circ}.}
    \label{fig:circ-det-obj}
\end{figure}
\begin{figure}[h!]
    \centering
    \includegraphics[width=\textwidth]{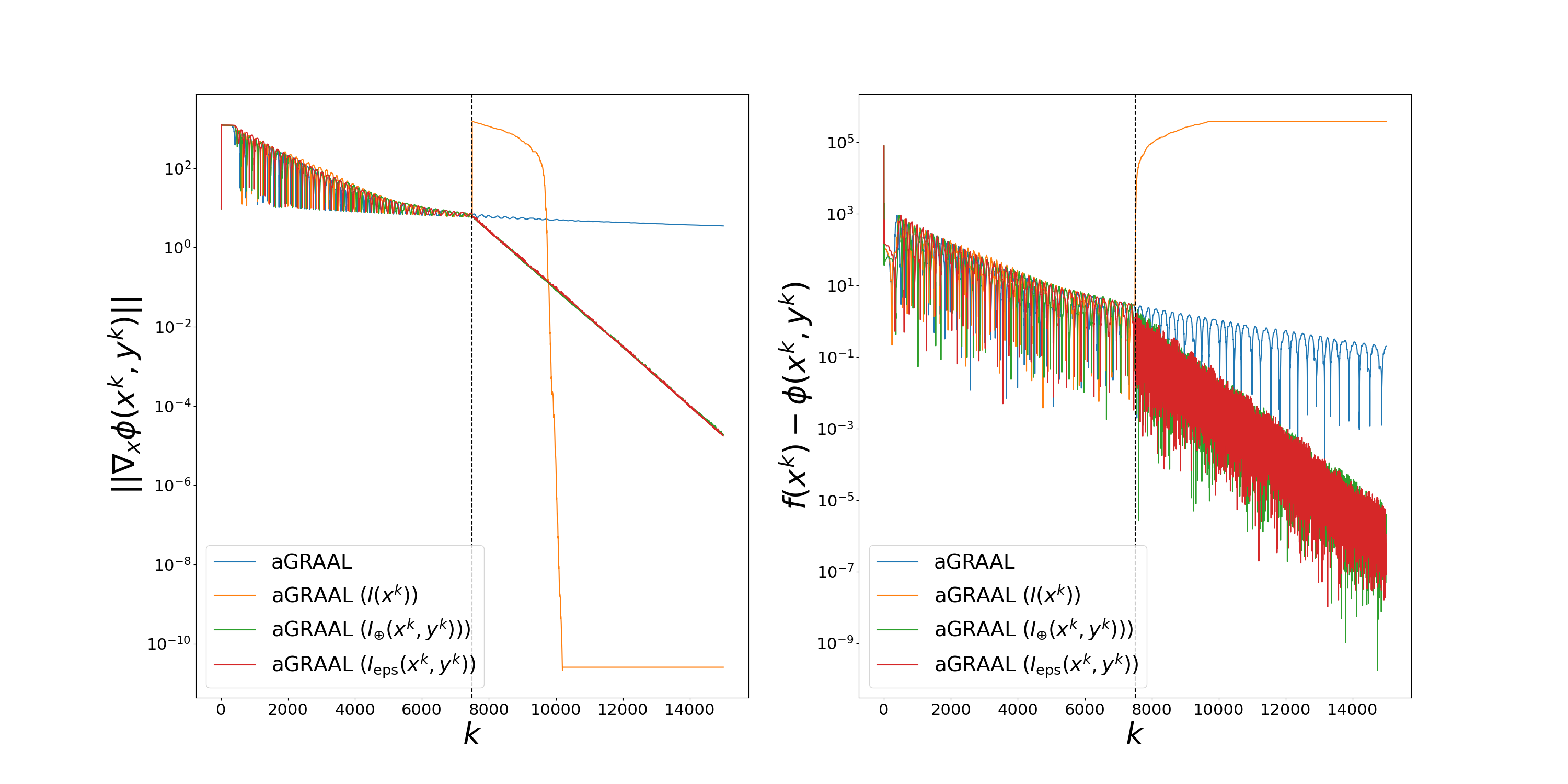}
    \caption{Error terms as derived in~Theorem~\ref{thm:reform} over iterations, for Algorithm~\ref{alg:det} applied to the generalised spanning circle problem~\eqref{eq:weighted circ}.}
    \label{fig:circ-det-err}
\end{figure}
\begin{figure}[h!]
    \centering
    \includegraphics[width=\textwidth]{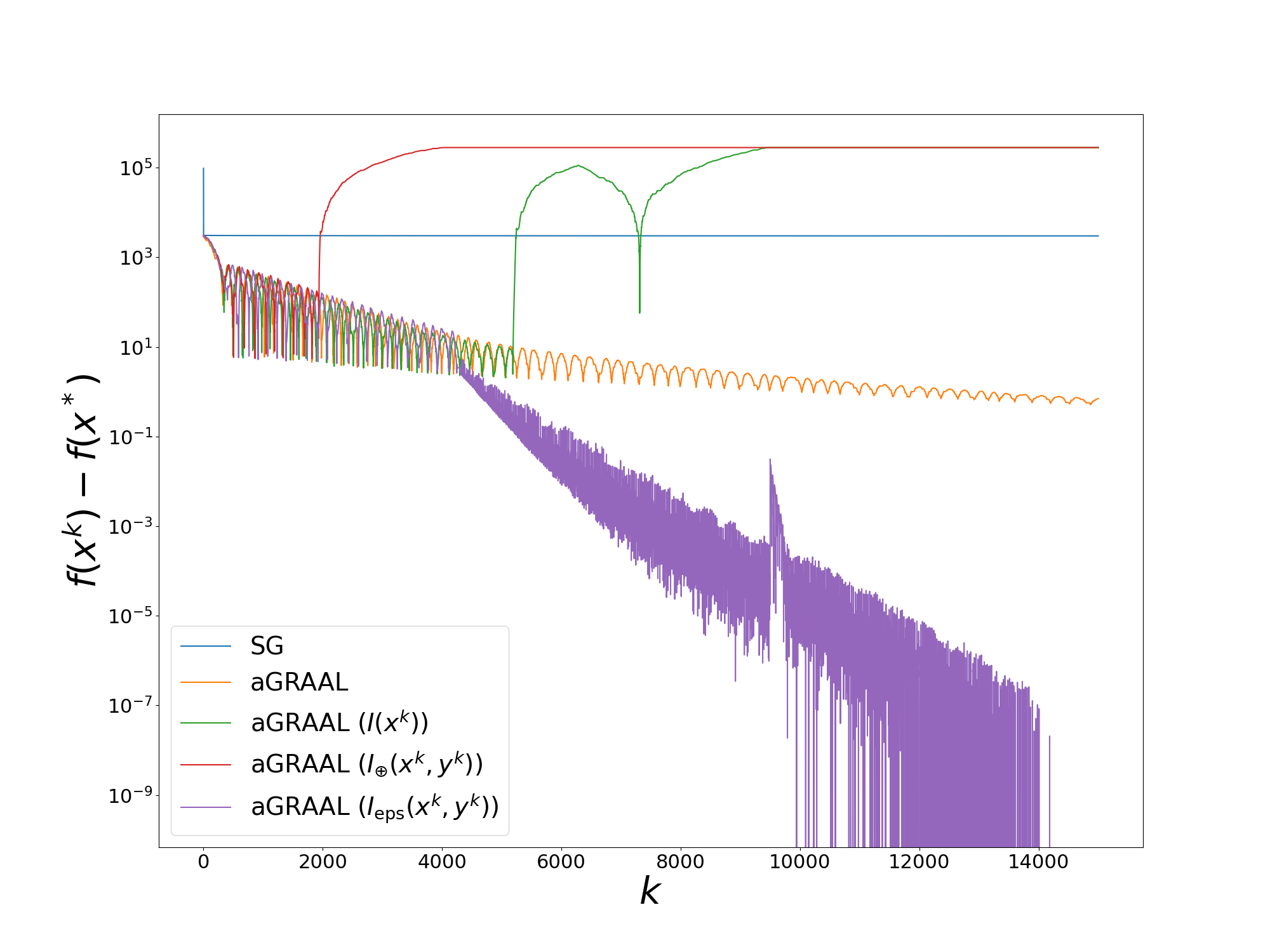}
    \caption{Objective function decay over iterations, for Algorithm~\ref{alg:det} applied to the generalised spanning circle problem~\eqref{eq:weighted circ}.}
    \label{fig:circ-stoc-obj}
\end{figure}
\begin{figure}[h!]
    \centering
    \includegraphics[width=\textwidth]{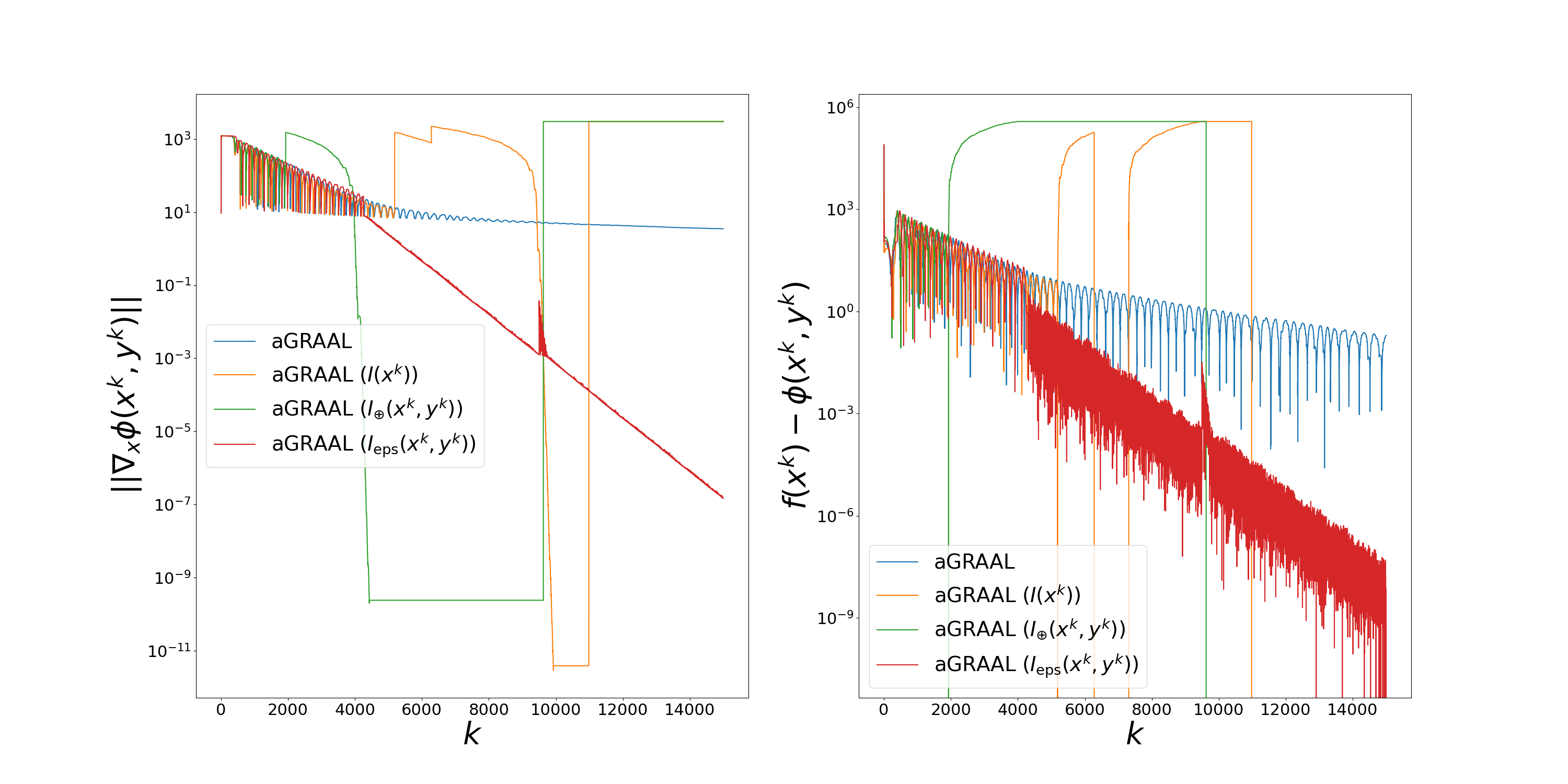}
    \caption{Error terms as derived in~Theorem~\ref{thm:reform} over iterations, for Algorithm~\ref{alg:stoc} applied to the generalised spanning circle problem~\eqref{eq:weighted circ}.}
    \label{fig:circ-stoc-err}
\end{figure}

Our final example is a generalisation of the spanning circle problem, which is to cover a given point set by a circle of minimal radius. By introducing weights $\omega\in\mathbb{R}^N$ and penalties $\kappa\in\mathbb{R}^N_{++}$, we study a more general facility location problem that is a combination of the problems studied in \cite[Case (b)]{elzinga1972geometrical} and \cite{hearn1982efficient}, namely the following
\begin{equation}\label{eq:weighted circ}
    \min_{x\in\mathbb{R}^n}\max_{i\in I}\omega_i\|x - p_i\|^2 + \kappa_i.
\end{equation}

In the standard case where all weights and penalties are equal, the problem can be interpreted simply as deciding a location for a central facility between customers $p_1,\dots,p_N\in\mathbb{R}^n$, so as to minimise the maximum of travel distances between all customers and the facility. The inclusion of weights into the problem models priorities toward certain customers, and penalties are some fixed travel cost in addition to the distance $\|x-p_i\|^2$. In our setting, we generated a random dataset of $3500$ points uniformly in $[-100,100]^2$, and clustered this dataset by applying the OPTICS algorithm from the sklearn package \cite{scikit-learn}, resulting in $N=1014$ clusters. From here, we set $p_i$ as the mean of the $i$'th cluster. We assigned $\omega_i$ as the population of the $i$'th cluster, so that more populous clusters are given higher priority. Finally, $\kappa_i$ is calculated in terms of the intracluster distance, and represents a fixed travel distance in addition to travel to/from the cluster.

In the deterministic setting shown in Figures~\ref{fig:circ-det-obj} and~\ref{fig:circ-det-err}, the measurements $I_{\text{eps}}$ and $I_{\oplus}$ provide a similar speedup of convergence as was observed in the examples of the previous sections. However, $I(x^{k_0})$ has once again failed to accurately measure $I(x^*)$, and the effect of this is that the sequence visibly converges to the wrong point after this correction step. Similarly for the stochastic versions shown in Figures~\ref{fig:circ-stoc-obj} and~\ref{fig:circ-stoc-err}, we observe that Algorithm~\ref{alg:stoc} performs well only by applying $I_{\text{eps}}$, which we assume is the result of $I_{\text{eps}}$ having higher accuracy in measuring $I(x^*)$ compared to $I$ and $I_{\oplus}$, although this could also be due to the latter measurements being taken too early.

\section{Conclusion}

In this paper, we presented a smooth min-max model~\eqref{eq:saddle} for solving the nonsmooth finite max problem~\eqref{eq:minmax}. We first justified the smooth model by proving its equivalence to the original problem, and described how we could solve it within the framework of variational inequalities. Then we presented several approaches for identifying the active function support $I(x^*)$ within finitely many iterations of any method. We then demonstrated how the concept of support identification can be applied to speed up convergence in a numerical setting. In our experiments, we compared the accuracy of the many support measurements given in Section~\ref{sec:identification}, and observed that our chosen algorithm converges faster when applied to the resulting simplification.

We conclude by presenting some directions for further research.

\begin{itemize}
    \item \textbf{Assumption Relaxation.} In practice, many optimisation problems are non-convex. However, many of the results derived here depend on convexity through Lemma~\ref{lem:subdiff}. Therefore, it is unclear whether something similar can be achieved, for example, with weak convexity instead of convexity.
    \item \textbf{Alternative Methods.} In this paper, we worked within the framework of the saddle point model~\eqref{eq:saddle} and the convergence results that can be derived there, but this is far from the only approach for solving~\eqref{eq:minmax}. One interesting case could be if the projections onto the epigraphs $\text{epi}_{f_i} = \{(x,t)\in\Hilbert\times\mathbb{R}\colon f_i(x)\leq t\}$ are available, then algorithms such as PDHG \cite{chambolle2016introduction, condat2023randprox, cohen2020optimal} could be applied.
    \item \textbf{Stochastic Solvers.} If $N$ is large, then~\eqref{eq:saddle} could also benefit from a stochastic method for solving~\eqref{eq:VIP}. However, in this case, the convergence results are slightly different from the deterministic setting (\emph{eg}, $(x^k,y^k)\to(x^*,y^*)$ almost surely).
    \item \textbf{Generalised Problems.} It would be useful to consider an extension~\eqref{eq:minmax} to include a proper, convex, and lsc function $g\colon\Hilbert\to(-\infty,+\infty]$, that is, the problem given by
    $$\min_{x\in\Hilbert}f(x) + g(x).$$
    This problem is more difficult to analyse since, for instance, we would no longer have $\dist(0,\partial f(x^k))\to 0$ in Theorem~\ref{thm:subgradient}. One other useful generalisation could be to take the \textit{sum} of finite max functions $f(x) = \sum_{j=1}^n\max\{f_{1,j}(x),\dots,f_{r,j}(x)\}$ (eg, the $\ell_1$-norm $\|x\|_1 = \sum_{j=1}^n \max\{-x,x\}$). Technically, these functions could be written in the form of~\eqref{eq:minmax}, but would require $r^n$ subfunctions. These extensions would enable us to study, for instance,
    \begin{align*}
        \text{(Support Vector Machines)}\quad & \min_{(w,w_0)\in\mathbb{R}^m\times\mathbb{R}} \|w\|^2_2 + C\sum_{j=1}^n \max\{0,1-y_i(\langle w,x_i\rangle - w_0)\}\\
        \text{(Basis Pursuit)}\quad & \min_{x\in\mathbb{R}^n}  \|x\|_1\quad\text{s.t.~} Ax=b.
    \end{align*}
\end{itemize}

\paragraph{Acknowledgements.}
MKT is supported in part by Australian Research Council grants DE200100063 and DP230101749. DJU is supported in part by Australian Research Council grant DE200100063.



\begin{thebibliography}{10}

\bibitem{Bauschke2017}
H.~H. Bauschke and P.~L. Combettes.
\newblock {\em Convex Analysis and Monotone Operator Theory in Hilbert Spaces}.
\newblock Springer International Publishing, Cham, 2017.

\bibitem{binev2017data}
P.~Binev, A.~Cohen, W.~Dahmen, R.~DeVore, G.~Petrova, and P.~Wojtaszczyk.
\newblock Data assimilation in reduced modeling.
\newblock {\em SIAM/ASA Journal on Uncertainty Quantification}, 5(1):1--29, 2017.

\bibitem{Borwein2006}
J.~Borwein and A.~Lewis.
\newblock {\em Convex Analysis and Nonlinear Optimization: Theory and Examples}.
\newblock Springer New York, New York, NY, 2006.

\bibitem{borwein2010convex}
J.~M. Borwein and J.~D. Vanderwerff.
\newblock {\em Convex Functions: Constructions, Characterizations and Counterexamples}, volume 172.
\newblock Cambridge University Press Cambridge, 2010.

\bibitem{Broendsted1983}
A.~Br{\o}ndsted.
\newblock {\em An Introduction to Convex Polytopes}.
\newblock Springer New York, New York, NY, 1983.

\bibitem{burke1988identification}
J.~V. Burke and J.~J. Mor{\'e}.
\newblock On the identification of active constraints.
\newblock {\em SIAM Journal on Numerical Analysis}, 25(5):1197--1211, 1988.

\bibitem{chambolle2016introduction}
A.~Chambolle and T.~Pock.
\newblock An introduction to continuous optimization for imaging.
\newblock {\em Acta Numerica}, 25:161--319, 2016.

\bibitem{cohen2020optimal}
A.~Cohen, W.~Dahmen, R.~DeVore, J.~Fadili, O.~Mula, and J.~Nichols.
\newblock Optimal reduced model algorithms for data-based state estimation.
\newblock {\em SIAM Journal on Numerical Analysis}, 58(6):3355--3381, 2020.

\bibitem{condat2016fast}
L.~Condat.
\newblock Fast projection onto the simplex and the $\pmb{l}_1$ ball.
\newblock {\em Mathematical Programming}, 158(1-2):575--585, 2016.

\bibitem{condat2023randprox}
L.~Condat and P.~Richt{\'a}rik.
\newblock Randprox: Primal-dual optimization algorithms with randomized proximal updates.
\newblock In {\em International Conference on Learning Representations (ICLR)}, 2023.

\bibitem{deodhare1996synthesis}
D.~Deodhare, M.~Vidyasagar, and S.~Sathiya~Keethi.
\newblock Synthesis of fault-tolerant feedforward neural networks using minimax optimization.
\newblock {\em IEEE Transactions on Neural Networks}, 9(5):891--900, 1998.

\bibitem{elzinga1972geometrical}
J.~Elzinga and D.~W. Hearn.
\newblock Geometrical solutions for some minimax location problems.
\newblock {\em Transportation Science}, 6(4):379--394, 1972.

\bibitem{elzinga1972minimum}
J.~Elzinga. and D.~W. Hearn.
\newblock The minimum covering sphere problem.
\newblock {\em Management Science}, 19:96--104, 09 1972.

\bibitem{elzinga1976minimax}
J.~Elzinga, D.~W. Hearn, and W.~D. Randolph.
\newblock Minimax multifacility location with {E}uclidean distances.
\newblock {\em Transportation Science}, 10(4):321--336, 1976.

\bibitem{facchinei1998accurate}
F.~Facchinei, A.~Fischer, and C.~Kanzow.
\newblock On the accurate identification of active constraints.
\newblock {\em SIAM Journal on Optimization}, 9(1):14--32, 1998.

\bibitem{facchinei2003finite}
F.~Facchinei and J.-S. Pang.
\newblock {\em Finite-Dimensional Variational Inequalities and Complementarity Problems}.
\newblock Springer, 2003.

\bibitem{gaudioso2006incremental}
M.~Gaudioso, G.~Giallombardo, and G.~Miglionico.
\newblock An incremental method for solving convex finite min-max problems.
\newblock {\em Mathematics of Operations Research}, 31:173--187, 2006.

\bibitem{Gauvin1977ANA}
J.~Gauvin.
\newblock A necessary and sufficient regularity condition to have bounded multipliers in nonconvex programming.
\newblock {\em Mathematical Programming}, 12:136--138, 1977.

\bibitem{hare2020derivative}
W.~Hare, C.~Planiden, and C.~Sagastizábal.
\newblock A derivative-free $\mathcal{VU}$-algorithm for convex finite-max problems.
\newblock {\em Optimization Methods and Software}, 35(3):521--559, 2020.

\bibitem{hare2004identifying}
W.~L. Hare and A.~S. Lewis.
\newblock Identifying active constraints via partial smoothness and prox-regularity.
\newblock {\em Journal of Convex Analysis}, 11(2):251--266, 2004.

\bibitem{hare2007identifying}
W.~L. Hare and A.~S. Lewis.
\newblock Identifying active manifolds.
\newblock {\em Algorithmic Operations Research}, 2(2):75--82, 2007.

\bibitem{hearn1982efficient}
D.~W. Hearn and J.~Vijay.
\newblock Efficient algorithms for the (weighted) minimum circle problem.
\newblock {\em Operations Research}, 30(4):777--795, 1982.

\bibitem{Hiriart-Urruty1993}
J.-B. Hiriart-Urruty and C.~Lemar{\'e}chal.
\newblock {\em Convex Analysis and Minimization Algorithms II: Advanced Theory and Bundle Methods}.
\newblock Springer Berlin Heidelberg, Berlin, Heidelberg, 1993.

\bibitem{kinderlehrer2000introduction}
D.~Kinderlehrer and G.~Stampacchia.
\newblock {\em An Introduction to Variational Inequalities and their Applications}.
\newblock SIAM, 2000.

\bibitem{kyparisis1985uniqueness}
J.~Kyparisis.
\newblock On uniqueness of {K}uhn-{T}ucker multipliers in nonlinear programming.
\newblock {\em Mathematical Programming}, 32:242--246, 1985.

\bibitem{liu2011maxmin}
Y.-F. Liu, Y.-H. Dai, and Z.-Q. Luo.
\newblock Max-min fairness linear transceiver design for a multi-user {MIMO} interference channel.
\newblock In {\em 2011 IEEE International Conference on Communications (ICC)}, pages 1--5, 2011.

\bibitem{love1973multi}
R.~F. Love, G.~O. Wesolowsky, and S.~A. Kraemer.
\newblock A multi-facility minimax location method for {E}uclidean distances.
\newblock {\em International Journal of Production Research}, 11(1):37--45, 1973.

\bibitem{malitsky2020golden}
Y.~Malitsky.
\newblock Golden ratio algorithms for variational inequalities.
\newblock {\em Mathematical Programming}, 184(1-2):383--410, 2020.

\bibitem{manevitz1996approx}
L.~Manevitz, M.~Shoham, and M.~Meltser.
\newblock Approximating functions by neural networks: A constructive solution in the uniform norm.
\newblock {\em Neural networks}, 9:965--978, 09 1996.

\bibitem{MANGASARIAN196737}
O.~Mangasarian and S.~Fromovitz.
\newblock The {F}ritz {J}ohn necessary optimality conditions in the presence of equality and inequality constraints.
\newblock {\em Journal of Mathematical Analysis and Applications}, 17(1):37--47, 1967.

\bibitem{mifflin2005VU}
R.~Mifflin and C.~Sagastizabal.
\newblock A $\mathcal{VU}$-algorithm for convex minimization.
\newblock {\em Mathematical Programming}, 104:583--608, 11 2005.

\bibitem{nagahara2011MinMax}
M.~Nagahara.
\newblock Min-max design of {FIR} digital filters by semidefinite programming.
\newblock In C.~Cuadrado-Laborde, editor, {\em Applications of Digital Signal Processing}, chapter~10. IntechOpen, Rijeka, 2011.

\bibitem{oberlin2006active}
C.~Oberlin and S.~J. Wright.
\newblock Active set identification in nonlinear programming.
\newblock {\em SIAM Journal on Optimization}, 17(2):577--605, 2006.

\bibitem{scikit-learn}
F.~Pedregosa, G.~Varoquaux, A.~Gramfort, V.~Michel, B.~Thirion, O.~Grisel, M.~Blondel, P.~Prettenhofer, R.~Weiss, V.~Dubourg, J.~Vanderplas, A.~Passos, D.~Cournapeau, M.~Brucher, M.~Perrot, and E.~Duchesnay.
\newblock Scikit-learn: Machine learning in {P}ython.
\newblock {\em Journal of Machine Learning Research}, 12:2825--2830, 2011.

\bibitem{Polak2003}
E.~Polak.
\newblock Smoothing techniques for the solution of finite and semi-infinite min-max-min problems.
\newblock In G.~Di~Pillo and A.~Murli, editors, {\em High Performance Algorithms and Software for Nonlinear Optimization}, pages 343--362. Springer US, Boston, MA, 2003.

\bibitem{polak2003algorithms}
E.~Polak, J.~Royset, and R.~Womersley.
\newblock Algorithms with adaptive smoothing for finite minimax problems.
\newblock {\em Journal of Optimization Theory and Applications}, 119:459--484, 2003.

\bibitem{rockafellar1970monotone}
R.~T. Rockafellar.
\newblock Monotone operators associated with saddle-functions and minimax problems.
\newblock In {\em Proceedings of Symposia in Pure Mathematics}, volume~18, pages 241--250. American Mathematical Society, 1970.

\bibitem{rockafellar1997convex}
R.~T. Rockafellar.
\newblock {\em Convex analysis}, volume~11.
\newblock Princeton University Press, 1997.

\bibitem{rockafellar2009variational}
R.~T. Rockafellar and R.~J.-B. Wets.
\newblock {\em Variational analysis}, volume 317.
\newblock Springer Science \& Business Media, 2009.

\bibitem{ruszczynski2006nonlinear}
A.~Ruszczyński.
\newblock {\em Nonlinear Optimization}.
\newblock Princeton University Press, 2006.

\bibitem{Shor1985}
N.~Z. Shor.
\newblock {\em Minimization Methods for Non-Differentiable Functions}.
\newblock Springer Berlin Heidelberg, Berlin, Heidelberg, 1985.

\bibitem{wright1993identifiable}
S.~J. Wright.
\newblock Identifiable surfaces in constrained optimization.
\newblock {\em SIAM Journal on Control and Optimization}, 31(4):1063--1079, 1993.

\bibitem{xingsi1992entropy}
L.~Xingsi.
\newblock An entropy-based aggregate method for minimax optimization.
\newblock {\em Engineering Optimization}, 18(4):277--285, 1992.

\bibitem{xu2001smoothing}
S.~Xu.
\newblock Smoothing method for minimax problems.
\newblock {\em Computational Optimization and Applications}, 20(3):267--279, 2001.

\bibitem{yamamoto2003optimal}
Y.~Yamamoto, B.~D. Anderson, M.~Nagahara, and Y.~Koyanagi.
\newblock Optimal {FIR} approximation for discrete-time {IIR} filters.
\newblock In {\em IEEE Signal Proc. Letters}, volume~10, pages 273--276, 2003.

\bibitem{ZAMIR2015947}
Z.~R. Zamir, N.~Sukhorukova, H.~Amiel, A.~Ugon, and C.~Philippe.
\newblock Convex optimisation-based methods for {K}-complex detection.
\newblock {\em Applied Mathematics and Computation}, 268:947--956, 2015.

\end{thebibliography}
\end{document}